\documentclass[11pt]{amsart}
\usepackage{multirow,bigdelim}
\usepackage{amsfonts}
\usepackage{dsfont}
\usepackage{amsmath}
\usepackage{mathrsfs}
\usepackage{todonotes}
\usepackage{hyperref}
\usepackage{multimedia}
\usepackage{graphicx}
\usepackage{indentfirst}
\usepackage{bm}
\usepackage{amsthm}
\usepackage{mathrsfs}
\usepackage{latexsym}
\usepackage{amsmath}
\usepackage{amssymb}
\usepackage{microtype}
\usepackage{relsize}
\usepackage{xcolor}
\usepackage{hyperref}

\DeclareSymbolFont{SY}{U}{psy}{m}{n}
\DeclareMathSymbol{\emptyset}{\mathord}{SY}{'306}

\theoremstyle{plain}

\newtheorem{thm}{Theorem}[section]
\newtheorem{cor}[thm]{Corollary}
\newtheorem{lem}[thm]{Lemma}
\newtheorem{prop}[thm]{Proposition}
\newtheorem{defn}[thm]{Definition}

\theoremstyle{definition}
\newtheorem{rem}[thm]{Remark}

\numberwithin{equation}{section}

\def\ra{\rightarrow}

\def\beq{\begin{eqnarray}}
\def\eeq{\end{eqnarray}}
\def\beqa{\begin{eqnarray*}}
\def\eeqa{\end{eqnarray*}}

\def\la{\langle}
\def\ra{\rangle}

\begin{document}
\title[Geometric Similarity invariants of Cowen-Douglas Operators]{Geometric Similarity invariants of Cowen-Douglas Operators }

\author{Chunlan Jiang, Kui Ji and Dinesh Kumar Keshari}
\curraddr[C. Jiang and K. Ji]{Department of Mathematics, Hebei Normal University,
Shijiazhuang, Hebei 050016, China} 
\curraddr{School of Mathematical Sciences, National Institute of Science Education and Research, Bhubaneswar, HBNI, Post-Jatni, Khurda, 752050, India}

\email[C. Jiang]{cljiang@hebtu.edu.cn}
\email[K. Ji]{jikui@hebtu.edu.cn, jikuikui@163.com}
\email[D. K. Keshari]{dinesh@niser.ac.in}

\thanks{
The first and the second authors were Supported by National Natural Science
Foundation of China (Grant No. 11831006). The third author was supported by INSPIRE Faculty Award [DST/INSPIRE/04/2014/002519], Department of Science and Technology (DST), India.
The third author thank the host for warm hospitality during the research visit to Department of Mathematics, Hebei Normal University.}



\begin{abstract}


In 1978, M. J. Cowen and R.G. Douglas introduced a class of operators $B_n(\Omega)$ (known as Cowen-Douglas class of operators) and associated a Hermitian holomorphic vector bundle to such an operator. 
They gave a complete set of unitary invariants in terms of the curvature and its covariant derivatives. At the same time they asked whether one can use geometric ideas to find a complete set of similarity invariants of Cowen-Douglas operators. We give a partial answer to this question. In this paper, we show that the curvature and the second fundamental form completely determine the similarity orbit of a norm dense class of Cowen-Douglas operators. As an application we show that uncountably many (non-similar) strongly irreducible operators in $B_n(\mathbb{D})$ can be constructed from a given operator in $B_1(\mathbb{D})$. We also characterize a class of strongly irreducible weakly homogeneous operators in $B_n(\mathbb{D})$.
\end{abstract}

\maketitle

\section{Introduction}

Let $\mathcal H$ be a complex separable Hilbert space, $\mathcal{L}(\mathcal{H})$ be the set of all bounded  linear operators on $\mathcal{H}$ and   Grassmannian $\mbox{Gr}(n,{\mathcal H})$ be the set of all $n$-dimensional subspaces of the Hilbert space ${\mathcal H}.$
For an  open bounded  connected subset $\Omega$  of the complex plane
$\mathbb{C},$ and $n\in \mathbb N,$ 
a map $t: \Omega \rightarrow \mbox{Gr}(n,{\mathcal H})$ is said to be a holomorphic curve, if there exist $n$ (point-wise linearly independent) holomorphic functions $ \gamma_1 ,\gamma_2 ,\cdots, \gamma_n$ on ${\Omega}$ taking values in $\mathcal H$ such that $t(w)=\bigvee \{\gamma_1(w),\cdots,\gamma_n(w)\},$ $w\in \Omega.$ 
Each holomorphic curve $t:\Omega \rightarrow \mbox{Gr}(n,{\mathcal H})$
gives rise to a rank $n$ Hermitian holomorphic vector bundle $E_t$ over $\Omega$, namely,
$$E_{t}=\{(x,w)\in {\mathcal H}\times \Omega \mid x\in t(w)\}
\,\,\mbox{and} \,\,\pi:E_t\rightarrow \Omega,\,\,\mbox{where}
\,\,\pi(x,w)=w.$$

In 1978, M. J. Cowen and R. G. Douglas introduce a class of  operators denoted by $B_n(\Omega)$ in a very influential paper \cite{CD}. An operator $T$ acting on  $\mathcal H$ is said to be a Cowen-Douglas operator with index $n$ associated with an open bounded subset $\Omega$  (or $T\in B_n(\Omega)$),  if  $T-w$ is surjective, $\dim \mbox{ker}(T-w)=n$ for all $w\in \Omega$ and $\bigvee\limits_{w\in \Omega} \mbox{ker}(T-w)=\mathcal{H}$.

M. J. Cowen and R. G. Douglas  showed that each operator $T$ in $B_n(\Omega)$  gives also rise to a rank $n$ Hermitian holomorphic vector bundle $E_T$ over $\Omega$, 
$$E_{T}=\{(x,w)\in {\mathcal H}\times \Omega \mid x\in \ker (T-w)\}
\,\,\mbox{and} \,\,\pi:E_T\rightarrow \Omega,\,\,\mbox{where}
\,\,\pi(x,w)=w.$$

Two holomorphic curves $t,\,\tilde{t}:
\Omega\rightarrow \mbox{Gr}(n,{\mathcal H})$  are said to be congruent if vector bundles $E_t$ and $E_{\tilde{t}}$ are locally equivalent as a Hermitian holomorphic vector bundle. Furthermore, $t$ and $\tilde{t}$ are said to be unitarily equivalent (denoted by $t\sim_u\tilde{t}$), if there exists a unitary operator $U (\in \mathcal{B}(\mathcal{H}))$ such that $U(w)t(w)=\tilde{t}(w),$  where $U(w):=U|_{E_t(w)}$ is the restriction of the unitary operator $U$ to the fiber $E_t(w)=\pi^{-1}(w)$. It is easy to see,  by using the Rigidity Theorem in \cite{CD}, $t$ and $\tilde{t}$ are congruent if and only if $t$ and $\tilde{t}$ are  unitarily equivalent.

 Two holomorphic curves  $t$ and
$\tilde{t}$ are said to be similarity equivalent (denoted by $t\sim_s \tilde{t}$) if there
exists an invertible operator $X\in {\mathcal B(\mathcal{H})}$ such that
$X(w)t(w)={\tilde{t}(w)}$, where $X(w):= X_{|{E_t}(w)}$ is the restriction of $X$ to the fiber $E_t(w)$. In this case, we say that the vector bundles $E_t$ and $E_{\tilde{t}}$ are similarity equivalent.

For an open   bounded connected subset $\Omega$ of $\mathbb{C}$, a Cowen-Douglas operator $T$ with index $n$  induces a non-constant 
holomorphic curve $t:\Omega \to \mbox{Gr}(n,{\mathcal H})$, namely,  $t(w)=
\ker(T-w), w\in \Omega$ and hence a Hermitian holomorphic vector bundle $E_t$ (here vector bundle $E_t$ is same as $E_T$). Unitary and similarity invariants for the operator $T$ are obtained from the vector bundle $E_T$.  


To describe these invariants, we need the curvature of the vector bundle $E_T$ along with its covariant derivatives. Let us recall some of these notions
following \cite{CD}.

The metric of bundle $E_T$ with respect to a holomorphic frame $\gamma$ is given by 
$$h_\gamma(w)=\big (\!\!\big (\langle \gamma_j(w),\gamma_i(w)\rangle \big ) \!\!\big )_{i,j=1}^n,\;
w\in \Omega,$$ where $\gamma(w)=\bigvee\{\gamma_1(w),\cdots,\gamma_n(w)\}, w\in \Omega.$

If we let $\bar{\partial}$ denote the complex
structure of the vector bundle $E_T,$ then the connection compatible
with both the complex structure $\bar{\partial}$  and the metric $h_{\gamma}$ is
canonically determined and is given by the formula $h_{\gamma}^{-1}(\tfrac{\partial }{\partial w} h_{\gamma})
\,dw.$ The curvature of the Hermitian holomorphic  vector bundle
$E_T$ is the $(1,1)$ form
$$ \bar{\partial}\big (h_\gamma^{-1}(w){\partial}
h_\gamma(w)\big )=-\frac{\partial}{\partial \overline{w}}\big (h_\gamma^{-1}(w)\frac{\partial}{\partial w}
h_\gamma(w)\big )\, dw \wedge d\bar{w}. \eqno$$ We let $\mathcal{K}_T(w)$ denote the coefficient of this $(1,1)$ form, that is,
$\mathcal{K}_T(w):= -\frac{\partial}{\partial \overline{w}}\big (h_\gamma^{-1}(w)\frac{\partial}{\partial w}
h_\gamma(w)\big )$ and we call it curvature.

Since the curvature $\mathcal{K}_T$ can be thought of  as a bundle map, following the definition of the covariant partial derivatives of bundle map,  
its  covariant partial derivatives $ \mathcal{K}_{T, w^i\overline{w}^j}$, $i,j\in
\mathbb{N}\cup \{0\}$ are given by: 

 (1)\,\,$\mathcal{K}_{T,w^i\overline{w}^{j+1}}=\frac{\partial}{\partial \overline{w}}(\mathcal{K}_{T,w^i\overline{w}^{j}});$

 (2)\,\,$\mathcal{K}_{T,w^{i+1}\overline{w}^j}=\frac{\partial}{\partial w}(\mathcal{K}_{T,w^{i}\overline{w}^j})+[h_{\gamma}^{-1}\frac{\partial}{\partial w}h_{\gamma},\mathcal{K}_{T,w^{i}\overline{w}^j}].$

The curvature and its covariant partial derivatives  are complete  unitary invariants of for an operator in the Cowen-Douglas class. M. J. Cowen and R. G. Douglas proved in \cite{CD} the following theorem.

\begin{thm}\label{CDT1}
Let $T$ and $\tilde{T}$ be two Cowen-Douglas operators with index n. Then $T$ and $\tilde{T}$ are  unitarily equivalent if
and only if there exists an isometric holomorphic bundle map $V: E_{T}\rightarrow E_{\tilde{T}}$
such that
$$V \big ( {\mathcal K}_{T,w^i\overline{w}^j} \big )=\big ( {\mathcal K}_{\tilde{T},w^i\overline{w}^j}\big ) V,\;
\, 0\leq i\leq j\leq i+j\leq n, \, (i,j)\neq (0,n),(n,0).$$
\end{thm}

In particular, if $T$ and $\tilde{T}$ are Cowen-Douglas operators with index one, then $T\sim_u \tilde{T}$ if and only if $\mathcal{K}_{T}=\mathcal{K}_{\tilde{T}}$.

 Theorem \ref{CDT1} says  that the local  complex geometric invariants  are 
of global nature from the point of view of  unitary equivalence. However, for similarity equivalence the global invariants are not 
easily detectable by the local invariants such as the curvature and its covariant derivatives.  It is because  that the  
holomorphic bundle map determined by invertible operators is not rigid. In other words, one does not know when a bundle map 
with local isomorphism can be extended to an invertible operator in $\mathcal{B}(\mathcal{H})$. In the absence of a characterization of equivalent classes under an invertible linear transformation, M. J. Cowen and R. G. Douglas made following conjecture in \cite{CD}.

{\bf Conjeture:} Let $\mathbb{D}$ denote the open unit disk. Let $T, \tilde{T} \in B_1(\mathbb{D})$ with the spectrums of $T$ and $\tilde{T}$ are closure of $\mathbb{D}$ (denoted by $\bar{\mathbb{D}}$). Then $T\sim_s \tilde{T}$ if and only if $$\lim\limits_{w\rightarrow \partial \mathbb{D}}\frac{\mathcal{K}_T(w)}{\mathcal{K}_{\tilde{T}}(w)}=1.$$

This conjecture turned out to be false (cf. \cite{ CM1, CM}).

The class of Cowen-Douglas operators is very rich. In fact,  the norm closure 
of Cowen-Douglas operators contains the collection of all quasi-triangular operators with spectrum being connected. This follows from the famous 
similarity orbit theorem given by C. Apostol, L. A. Fialkow, D. A. Herrero and D. Voiculescu (See \cite{SO}).  Subsequently, the Cowen-Douglas operator has been one of the important ingredients in the research of operator theory(cf \cite{Al, CA, CM, Chen, EL1, EL2, Jig, JLY, KS, LS, MS, MJ, NY, RR, WZ, DW, YGL}).

%
%
%
%

To find similarity invariant for Cowen-Douglas operators in terms of geometric invariants, we need following concepts and theorems. 


\begin{thm}[{Upper triangular representation theorem}, \cite{JW}] 


 Let $T\in\mathcal{L}(\mathcal{H})$ be a Cowen-Douglas operator with index $n$,
 then there exists a orthogonal decomposition $\mathcal{H}=\mathcal{H}_1\oplus\mathcal{H}_2\oplus\cdots\oplus\mathcal{H}_n$ and  operators $T_{1,1}, T_{2,2}, \cdots, T_{n,n}$ in $B_1(\Omega)$ such that $T$ takes following form 

\begin{equation} \label{1.1T}
T=\left ( \begin{smallmatrix}T_{1,1} & T_{1,2}& T_{1,3}& \cdots & T_{1,n}\\
0&T_{2,2}&T_{2,3}&\cdots&T_{2,n}\\
\vdots&\ddots&\ddots&\ddots&\vdots\\
0&\cdots&0&T_{n-1,n-1}&T_{n-1,n}\\
0&\cdots&\cdots&0&T_{n,n}
\end{smallmatrix}\right ).
\end{equation}

\renewcommand\arraystretch{1}\noindent

Let 
$\{\gamma_1,\gamma_2,\cdots,\gamma_{n}\}$ be a holomorphic frame
of  $E_T$ with ${\mathcal H}=
\bigvee\{\gamma_i(w),\;w\in \Omega,\;1\leq i \leq n\}, $  and   $t_i:\Omega \to
\mbox{Gr}(1,{\mathcal H}_i)$  be a holomorphic frame of $E_{T_{i,i}},\,1\leq i \leq n$. Then we can find certain 
relation between $\{\gamma_i\}^{n}_{i=1}$ and $\{t_i\}^{n}_{i=1}$ as the following equations:
\begin{equation}\label{1.1}\begin{array}{llll}\gamma_1&=&t_1\\
                   \gamma_{2}&=&\phi_{1,2}(t_2)+t_{2}\\
                   \gamma_{3}&=&\phi_{1,3}(t_3)+\phi_{2,3}(t_3)+t_{3}\\
                  \cdots &=& \cdots \cdots\cdots\cdots\cdots\cdots\cdots\cdots\\
                    \gamma_{j}&=&\phi_{1,j}(t_j)+\cdots+\phi_{i,j}(t_j)+\cdots+t_{j}\\
                    \cdots &=& \cdots \cdots\cdots\cdots\cdots\cdots\cdots\cdots\cdots\cdots\cdots\\
                   \gamma_{n}&=&\phi_{1,n}(t_{n})+\cdots+\phi_{i,n}(t_{n})+\cdots+t_{n},\\
\end{array}\end{equation}
where $\phi_{i,j}, (i,j=1,2,\cdots,n)$ are certain holomorphic bundle maps.

\end{thm}

We would expect that these bundle maps might reflect geometric similarity invariants of the operator $T$. However, it is far from enough on the coarse
relation of above equations. In particular, in order to use geometric terms such as curvature and second fundamental form for similarity 
invariants of operator in $(1.1)$,  we have to give them more internal structures. For example, we may assume that $T_{i,i+1}$ are non zero and intertwines $T_{i,i}$ and $T_{i+1,i+1}$, that is, $T_{i,i}T_{i,i+1}=T_{i,i+1}T_{i+1,i+1}, 1\leq i\leq n-1$ (it is denoted by $\mathcal{FB}_n(\Omega)$ (see in \cite{JJKMCR})).

For a $2\times 2$ block 
$\begin{pmatrix}
T_{i,i}&T_{i,i+1}\\
0&T_{i+1,i+1}\\
\end{pmatrix}$, if  $T_{i,i}T_{i,i+1}=T_{i,i+1}T_{i+1,i+1},$  the corresponding second fundamental form $\theta_{i,i+1}$, which is obtained by R. G. Douglas and G. Misra (see in \cite{DM}), is 
\begin{equation} \label{sf}
\theta_{i,i+1}(T)(w) =  \frac{\mathcal K_{T_{i,i}}(w) \,d\bar{w}} {\big
(\frac{\|T_{i,i+1}(t_{i+1}(w))\|^2}{\|t_{i+1}(w)\|^2} - \mathcal
K_{T_{i,i}}(w)\big )^{1/2}}.\end{equation}

Let  $T,\widetilde{T}$ have upper-triangular representation as in $(1.1)$ and assume that $T_{i,i}, T_{i+1,i+1}$ and $\tilde{T}_{i,i}, \tilde{T}_{i+1,i+1}$ have intertwinings $T_{i,i+1}$ and $\tilde{T}_{i,i+1}$, respectively.  If $\mathcal{K}_{T_{i,i}}=\mathcal{K}_{\widetilde{T}_{i,i}}$, then from $(1.3)$ it is easy to see that $$\theta_{i,i+1}(T)(w) =\theta_{i,i+1}(\tilde{T)}(w) \Leftrightarrow \frac{\|T_{i,i+1}(t_{i+1}(w))\|^2}{\|t_{i+1}(w)\|^2}=\frac{\|\widetilde{T}_{i,i+1}(\widetilde{t}_{i+1}(w))\|^2}{\|\widetilde{t}_{i+1}(w)\|^2}.$$

If upper-triangular representation in $(1.1)$ has such a good internal structure, then a complete set of unitary (or similarity) invariants of $T$ can be obtained in terms of  the curvature and the second fundamental form. 

In this  paper, we  introduce a subset of Cowen-Douglas operators denoted by $\mathcal{CFB}_n(\Omega)$ (See Definition \ref{CFB1}). We prove that $\mathcal{CFB}_n(\Omega)$ is norm dense in 
$B_n(\Omega)$(See Proposition \ref{ND}). Hence it is meaningful to discuss the geometric similarity invariants for operators belong to $ \mathcal{CFB}_n(\Omega)$. 
We would like to point that similarity  results for quasi-homogeneous Cowen-Douglas operators have been discussed in \cite{JJM}. However, the class  $ \mathcal{CFB}_n(\Omega)$ is quite different from the class of quasi-homogeneous Cowen--Douglas operators. For example, if we take the operator $T$ in 
$ \mathcal{CFB}_n(\Omega)$ such that $T_{i,i}=T_{i+1,i+1}, 1\leq i\leq n-1$ or $T_{i,i}$ is not a homogeneous operator, then $T$ is not a quasi-homogeneous Cowen-Douglas operator. Also, if we take a homogeneous Cowen-Douglas operator $T$ (which is quasi-homogeneous by definition) in $B_n(\Omega)$ for $n\geq 3$, then $T$ does not belong to $ \mathcal{CFB}_n(\Omega)$.

%
%
%
%
%

Roughly speaking, for operators $T$ in $\mathcal{CFB}_n(\Omega)$, 
  the curvature and the second fundamental form give a complete set of similarity invariants. The authors joint with  G. Misra gave a complete set of unitary invariants in terms  the curvature and the second fundamental form for operators in $\mathcal{FB}_n(\Omega)$ (See in \cite{JJKM}).

In the general case, 
the first and the second authors obtained a complete set of similarity invariants by using ordered $K_0$-group. However, ordered $K_0$-group is an algebraic invariant. So one might expect these algebraic invariants may give some insight in search of geometric invariants.
%
%


Recently,  R. G. Douglas, H. Kwon, S. Treil \cite{KT, DKS} and K. Ji \cite{HJK} used curvature to describe similarity invariants for a subclass of Cowen-Douglas operators. 

\begin{thm}\cite{DKS, HJK}\label{DKS}
 Let $T \in B_n(\mathbb{D})$ be an $m$-hypercontraction and $S_z$ is the multiplication on the weighted Bergman space. 
Then  T is similar to  
$\bigoplus\limits^n_{i=1}S^*_{z}$
if and only if there exists a bounded subharmonic function 
 $\psi$  defined on
$\mathbb{D}$ such that $$\mbox{trace}\,(\mathcal{K}_T)-\mbox{trace}\,(\mathcal{K}_{S^*_z})=\Delta \psi.$$

\end{thm}

 An operator $T$ is said to be homogeneous if $\phi_{\alpha}(T)$ is unitarily equivalent to
$T$ for each  M\"{o}bius transformation $\phi_{\alpha}$. 
G. Misra   proved the following theorem:
\begin{thm}[\cite{KM}]
Let $T_1$ and $T_2$ be two homogeneous Cowen-Douglas
operators with index one. Then $T_1$ is similar to $T_2$ if and only if $T_1$ is unitarily equivalent to $T_2$, that is, $\mathcal{K}_{T_1}=\mathcal{K}_{T_2}$. 
\end{thm}


The first and the second authors jointly with G. Misra extended the concepts of homogeneous operators to quasi-homogeneous operators as follows.

\begin{defn}[\cite{JJM}]Let $T\in \mathcal{FB}_n(\Omega)$ and $T$ has an $n\times n$ upper-triangular matrix in $(1.1)$.  Then the operator $T$ is called a quasi-homogeneous operator, i.e. $T\in \mathcal{QB}_n(\Omega)$, 
if  each $T_{i,i}$ is a  homogeneous operator in $B_1(\Omega)$  and $$T_{i,j}(t_j)\in \bigvee\{t^{(k)}_i, k\leq j-i-1\}.$$

\end{defn}

For the quasi-homogeneous operators,  the curvature and the second fundamental form completely describe similarity invariants.  

\begin{thm}[\cite{JJM}]
Let $T, S\in \mathcal{QB}_n(\Omega)$, then
\begin{equation}\nonumber 
  \left\{\begin{array}{lll}
\mathcal{K}_{T_{i,i}}=\mathcal{K}_{\widetilde{T}_{i,i}}\\
 \theta_{i,i+1}(T)=\theta_{i,i+1}(\widetilde{T})\\
   \end{array}
\right. \Longrightarrow T\sim_s\widetilde{T} ~~\mbox{if and only if}~~ T=\widetilde{T}.
 \end{equation}

\end{thm}

We point out that even if $T$ is a Cowen-Douglas
operator with index one, its spectral picture is also very complicated. The following theorem is due to D. A. Herreor shows its complexity. 

\begin{thm}[\cite{Her}]
Let $T\in \mathcal{L}(\mathcal{H})$ be a quasi-triangular operator with connected spectral picture. If there exists a point $w$ in the Fredholm domain of $T$ such that 
$\mbox{ind}\,(T-w)=1$. Then for any $\epsilon>0$, there exists a compact operator $K$ with $\|K\|<\epsilon$ such that $T+K$ is a Cowen-Douglas operator with index one.
\end{thm}

It is due to the complexity of structure of Cowen-Douglas operators and the fact that invertible operator is not an isometric bundle map.  Therefore for any two Cowen-Douglas operators $T$ and $\tilde{T}$ with index one,  to understand similarity between $T$ and $\tilde{T}$, we have to further explore the relation between $\mathcal{K}_T$ and $\mathcal{K}_{\tilde{T}}$.

We now summarize the content of this paper. In section 2, we introduce a subclass of Cowen-Douglas operators denoted by ${\mathcal{CFB}_n(\Omega)}$. We prove this class of operators are norm dense in the set of all of the Cowen-Douglas operators.  In section 3, we study an important property named by Property $(H)$. In section 4,  we  show that the curvature and the second fundamental form completely characterize the similarity invariants for all Cowen-Douglas operators in $\mathcal{CFB}_n(\Omega)$. In section 5, we characterize a class of weakly homogeneous operator in $B_n(\mathbb{D})$. We also construct uncountably many   strongly irreducible operators (non similar) in $B_n(\mathbb{D})$ from a given operator in $B_1(\mathbb{D})$.

\newpage

\section{The operator class $\mathcal{CFB}_n(\Omega)$}

In this section, we introduce  a subclass of class of Cowen-Douglas operators which is denoted as $\mathcal{CFB}_n(\Omega)$.  We show that $\mathcal{CFB}_n(\Omega)$ is norm dense in $B_n(\Omega)$. We first recall the definition of the subclass $\mathcal{FB}_n(\Omega)$ of $B_n(\Omega)$. This class has been studied in detail in \cite{JJKM}.

\begin{defn} $\mathcal{FB}_n(\Omega)$ is  the set of all bounded linear operators $T$ defined on some complex separable Hilbert space $\mathcal H = \mathcal H_0 \oplus \cdots \oplus \mathcal H_{n-1},$ which are of the form
$$T=\left ( \begin{matrix}
T_{1,1} & T_{1,2} &\cdots&T_{1,n}\\
&T_{2,2}&\cdots&T_{2,n} \\
&&\ddots&\vdots\\
&&&T_{n,n}\\
\end{matrix} \right )$$
where the operator $T_{i,i}:\mathcal H_i \to \mathcal H_i,$ defined on a complex separable Hilbert space $\mathcal H_i,$ $1\leq i \leq n,$ is assumed to be in $B_1(\Omega)$ and  $T_{i,i+1}:\mathcal H_{i+1} \to \mathcal H_i,$ is assumed to be a non-zero intertwining bounded operator, namely, $T_{i,i}T_{i,i+1}=T_{i,i+1}T_{i+1,i+1},$  $1\leq i \leq n-1$.
\end{defn}

To define the class $\mathcal{CFB}_n(\Omega)$, we need following definitions.

\begin{defn} Let $T_1$ and $T_2$ be bounded linear operators on $\mathcal{H}$. Rosenblum operators  $\tau_{T_1,T_2}$ and $\delta_{T_1}$ are maps from $\mathcal{L}(\mathcal{H})$ to $\mathcal{L}(\mathcal{H})$ defined as follows:

$$\tau_{T_1,T_2}(X)=T_{1}X-XT_{2} $$ 

and

$$\delta_{T_1}(X)=T_{1}X-XT_{1}$$

where $X$ is a bounded linear operator on $\mathcal{H}$. 
\end{defn}

\begin{defn}\label{ph}({\bf{Property (H)}})
 Let $T_1$ and $T_2$ be bounded linear operators on $\mathcal{H}$. We say that $T_1,T_2$ satisfy the \textit{Property $(H)$} if  the following condition holds:
if $X$ is a bounded linear operator defined on $\mathcal{H}$ such that  
$T_{1}X=XT_{2}$ and 
$X=T_{1}Z-ZT_{2},\; \mbox{for some}\; Z$ in $\mathcal{L}(\mathcal{H})$, then $X=0$.


\end{defn}

\begin{rem}
Let $T$ be an operator in $B_1(\Omega)$. If we take $T_1=T_2=T$, then from \cite{JJKM} it follows that $T_1,T_2$ satisfy the Property (H). 
\end{rem}
 
 We recall that $\{T\}^{\prime}$ denotes the commutant, that is, $\{T\}^{\prime}$ is the set of all of bounded linear operators which commute with $T$.
 
\begin{defn} Let $T$ be a bounded linear operator on $\mathcal{H}$. We say that $T$ is strongly irreducible if there is no non trivial idempotent in $\{T\}^{\prime}.$

\end{defn}

\begin{lem}\label{SPH}Let $T_1$ and $T_2$ be bounded linear operators on $\mathcal{H}$. Suppose that $S_1$ and $S_2$ are similar to $T_1$ and $T_2$, respectively. If $T_1, T_2$ satisfy the Property $(H)$, then $S_1,S_2$ also satisfy the Property $(H)$.

\end{lem}

\begin{proof} Since $T_i$ is similar to $S_i$, so there exists an invertible operator $X_i$ such that $X_iT_i=S_iX_i$, $1\leq i\leq 2$.   Let $Y$ be a bounded linear operator such that  $S_1Y=YS_2$ and $Y=S_1Z-ZS_2\; \mbox{for some}\; Z$ in $\mathcal{L}(\mathcal{H})$.
It is easy see that  
$$T_1X^{-1}_1YX_2=X^{-1}_1YX_2T_2$$
and 
\begin{eqnarray*}
X^{-1}_1YX_2&=&X^{-1}_1S_1X_1 X^{-1}_1ZX_2-X^{-1}_1ZX_2X^{-1}_2S_2X_2\\
&=&T_1X^{-1}_1ZX_2-X^{-1}_1ZX_2 T_2.
\end{eqnarray*}
Since $T_1,T_2$ satisfy the Property $(H)$, so $X^{-1}_1YX_2=0$ and hence $Y=0$.  Thus $S_1,S_2$ also satisfy the Property $(H)$. This completes the proof.

\end{proof}

\begin{defn}\label{CFB1}
A Cowen-Douglas operator $T$ with index $n$ is said to be in $\mathcal{CFB}_n(\Omega)$, if $T$ satisfies the following properties: 
\begin{itemize}
\item[(1)] $T$ can be written as an $n\times n$ upper-triangular matrix form under a topological direct decomposition of $\mathcal{H}$ and $diag\{T\}:=T_{1,1}\dotplus T_{2,2}\dotplus \cdots \dotplus T_{n,n}\in \{T\}^{\prime}$. Furthermore, each entry $$T_{i,j}=\phi_{i,j}T_{i,i+1}T_{i+1,i+2} \cdots T_{j-1,j}$$ where $\phi_{i,j}\in\{T_{i,i}\}^{\prime}$;
\item[(2)] $T_{i,i}, T_{i+1,i+1}$ satisfy the Property $(H)$, that is,  $\ker\,\tau_{T_{i,i},T_{i+1,i+1}}\cap \mbox{ran}\,\tau_{T_{i,i},T_{i+1,i+1}}=\{0\}$, $1\leq i\leq n-1$. 

\end{itemize}
\end{defn}

In order to prove the class  $\mathcal{CFB}_n(\Omega)$ is norm dense in $B_n(\Omega)$, we need the following concepts and lemmas:

\begin{defn}[Similarity invariant set]
Let $\mathcal{F}=\{A_{\alpha}\in \mathcal{L}(\mathcal{H}), \alpha\in \Lambda\}.$ 
We say $\mathcal{F}$ is a similarity invariant set, if  for any invertible operator $X$, 
$$X\mathcal{F}X^{-1}=\{XA_{\alpha}X^{-1}:A_{\alpha}\in \mathcal{F}\}=\mathcal{F}.$$

\end{defn}

\begin{defn}\cite{JW} If ${\mathcal K}(\mathcal{H})$ denotes the set of all compact operators acting on ${\mathcal H}$ and 
$\pi: {\mathcal L}(\mathcal {H})\rightarrow  {\mathcal L}(\mathcal {H})/{\mathcal K}(\mathcal{H})$ is the projection of ${\mathcal L}(\mathcal {H})$
onto the Calkin algebra, then $\sigma_{e}(T)$, the essential spectrum of $T$, is the spectrum of $\pi(T)$ in $ {\mathcal L}(\mathcal {H})/{\mathcal K}(\mathcal{H})$
and $\mathbb{C}\backslash \sigma_{e}(T)$ is called the Fredholm domain of $T$ and is denoted by $\rho_{F}(T)$.  Thus, $\sigma_{e}(T)=\sigma_{le}(T)\cup \sigma_{re}(T)$, 
where $\sigma_{le}(T)=\sigma_{l}(\pi(T))$ (left essential spectrum of $T$) and $\sigma_{re}(T)=\sigma_{r}(\pi(T))$ (right essential spectrum of $T$). 

On the other hand, the intersection $\sigma_{lre}(T):=\sigma_{le}(T)\cap \sigma_{re}(T)$ is called Wolf spectrum and it include the boundary $\partial \sigma_{e}(T)$ of $\sigma_{e}(T)$. Therefore, it is a non-empty compact subset of $\mathbb{C}$. Its complement $\mathbb{C}\backslash \sigma_{lre}(T)$ coincides with $\rho_{s-F}(T):=\{w\in \mathbb{C}: T-w$ is semi-Fredholm $\}$.  $\rho_{s-F}(T)$ is the disjoint union of the (possibly empty) open sets $\{\rho^n_{s-F}(T): -\infty\leq n\leq +\infty\},$ where 
$$\rho^n_{s-F}(T)=\{w\in \mathbb{C}:~~ \mbox{T-w is semi-Fredholm with}~~ \mbox{ind}\,(T-w)=n\}.$$

The spectrum picture of $T$, denoted by $\Lambda(T)$, is defined as the compact set $\sigma_{lre}(T)$, plus the data corresponding to the indices of $T-w$ for $w$ in the bounded components of $\rho_{s-F}(T).$ 

\end{defn}

\begin{lem}[\cite{Her}]\label{Herlemma0} Let $T\in  B_n(\Omega)$. Then 
$\sigma_{p}(T^*)=\emptyset,$ and $\sigma(T)$ is connected, where $\sigma_p(T^*)$ denotes the point spectrum of 
$T^*$. 

\end{lem}

\begin{lem}[\cite{Her}]\label{Herlemma} Let $T\in  \mathcal{L}(\mathcal{H})$, $\varepsilon>0$, and let $T$ be a quasi-triangular operator such that 
\begin{enumerate}
\item[(i)] $\sigma(T)$ is connected;
\item[(ii)] $\mbox{ind}\,(T-w)>0$, $w\in \rho_{F}(T)$;
\item[(iii)] there exist a  positive integer $n$ and $w_0\in \rho_{F}(T)$ such that $\mbox{ind}\,(T-w_0)=n,$
\end{enumerate}
then there exists a compact operator $K$ with $\|K\|<\varepsilon$ such that $T+K\in B_n(\Omega)$, where $w_0\in \Omega\subset \rho_{F}(T).$

\end{lem}

\begin{lem}[Voiculescu Theorem, \cite{Voi}]\label{vcthm} Let $T\in \mathcal{L}(\mathcal{H})$ and $\rho$ be a unital faithful $*$-representation of a separable 
$C^*$-subalgebra of the Calkin algebra containing the canonical image $\pi(T)$ and $\pi(I)$, on a separable space $\mathcal{H}_{\rho}.$ Let 
$A=\rho(\pi(T))$ and $k$ be a positive integer. Given $\varepsilon$, there exists $K \in \mathcal{K}(\mathcal{H})$, with $\|K\|<\varepsilon$, such that 
$$T-K\sim_u T\oplus A^{(\infty)}\sim_u T\oplus A^{(k)},$$
where $A^{(k)}$ denotes $\bigoplus\limits^{k}_{i=1}A$, and $A^{(\infty)}$ denotes $\bigoplus\limits^{\infty}_{i=1}A$. 

\end{lem}

\begin{lem}[Special case of similarity orbit Theorem, Apostle, Fialkow, Herrero, and Voiculescu\cite{SO}]\label{SOT}
Let $T$ and $S$ be in $B_n(\Omega)$ satisfy the following conditions: 
\begin{enumerate}
\item[(i)] $\sigma_{lre}(T)=\sigma_{lre}(S)$ and $\sigma_{lre}(T)$ is a perfect set; 
\item[(ii)]$\Lambda(T)=\Lambda(S).$
\end{enumerate}
Then there exist two sequences of 
invertible operators $\{X_n\}^{\infty}_{n=1}$ and $\{Y_n\}^{\infty}_{n=1}$ such that 
$$\lim\limits_{n\rightarrow \infty}X_nTX^{-1}_n=S,\;\;\;\;\;\; \lim\limits_{n\rightarrow \infty}Y_nSY^{-1}_n=T.$$

\end{lem}

\begin{lem} \label{JWdslemma}
Let $T\in  B_n(\Omega)$, $\varepsilon>0$ and $\Phi$ be an analytic Cauchy domain satisfying 
$$\sigma_{lre}(T)\subset \Phi \subset [\sigma_{lre}(T)]_{\varepsilon}:=\{w\in \mathbb{C}:\; \mbox{dist}\,(w, \sigma_{lre}(T))< \varepsilon\}, $$
then there exists $T_{\varepsilon}\in B_n(\Omega_1)$ such that 
\begin{enumerate}
\item[(i)] $\sigma_{lre}(T_{\varepsilon})=\Phi$ and $\Omega_1$ is an open connected subset of $\Omega$; 
\item[(ii)] $\|T-T_{\varepsilon}\|<\varepsilon.$
\end{enumerate}

\end{lem}

\begin{proof} Let $\delta=\mbox{dist}\,(\Phi, \partial  [\sigma_{lre}(T)]_{\varepsilon})$. By Lemma \ref{vcthm}, there exist an operator $A$ and an operator $K_1$ with $\|K_1\|<\frac{\delta}{3}$,
and a unitary operator $U$ such that 
$$U(T+K_1)U^*=A\oplus T,~\mbox{where}~\sigma_{lre}(A)=\sigma_{lre}(T).$$
By Theorem 1.25 (\cite{JW}) and Proposition 1.22 (\cite{JW}), there exists  a compact operator $K_2$ with $\|K_2\|<\frac{\delta}{3}$ such that 
$$U(T+K_1+K_2)U^*=\left ( \begin{matrix} N&A_{1,2}\\0&A_{\infty}\end{matrix}\right)\oplus T=B\oplus T,$$
and $\sigma_{e}(B)=\sigma(B)=\sigma_{e}(T)$, $\sigma_{lre}(B)=\sigma_{lre}(T)$, $\sigma(A_{\infty})=\sigma_e(A_{\infty})=\sigma_e(T)$, where 
$B=\left ( \begin{matrix} N&A_{1,2}\\0&A_{\infty} \end{matrix}\right)$ and 
$N$ is a diagonal normal operator of uniform infinite multiplicity and $\sigma(N)=\sigma_{lre}(T).$

Let $L=\left ( \begin{matrix} M&A_{1,2}\\0&A_{\infty}\end{matrix}\right)\oplus T,$ where $M$ is a diagonal normal operator of uniform infinite multiplication operator with $\sigma(M)=\sigma_{e}(M)=\Phi.$  By a direct calculation, we can see that 
$$\|L-U(T+K_1+K_2)U^*\|<\varepsilon-\delta.$$
Let  $T^{\prime }_{\varepsilon}$ denote $U^*LU$. By Lemma \ref{Herlemma0}, we have that $\sigma_p(T^{*})=\emptyset$, then it follows that $\sigma(T_{\varepsilon}^{\prime})$ is connected and $\sigma_{p}(T^{\prime *}_{\varepsilon})=\emptyset.$  When $\varepsilon$ is small enough, we can find $w\in \Omega$ and $\delta_1>0$ such that 
$\Omega_1=O_{w, \delta_1}$, the neighbourhood of $w$ such that $\Omega_1\subset \Omega$. 

Applying Lemma \ref{Herlemma} to $T^{\prime}_{\varepsilon}$, there is a compact operator $K_3$ with $\|K_3\|<\varepsilon$ such that $T_{\varepsilon}=T^{\prime}_{\varepsilon}+K_3\in B_n(\Omega_1)$, then $T_{\varepsilon}$ satisfies all the requirements of the lemma.

\end{proof}

\begin{lem}\cite{CD} \label{cdl1}Let $T\in B_n(\Omega)$ and $\Omega_1\subseteq \Omega$ be an open connected set.  Then $B_n(\Omega)\subseteq B_n(\Omega_1).$

\end{lem}

\begin{prop}\label{ND} $\mathcal{CFB}_n(\Omega)$ is norm dense in 
$B_n(\Omega)$. 
\end{prop}

\begin{proof} First, we shall prove that  $\mathcal{CFB}_n(\Omega)$ is a similarity invariant set.  For any $T\in \mathcal{CFB}_n(\Omega)$, by Definition \ref{CFB1}, there exist $n$ idempotents $\{P_i\}^n_{i=1}$ such that 
\begin{enumerate}
\item[(1)]$\sum\limits_{i=1}^nP_i=I, P_iP_j=0, i\neq j;$
\item[(2)]$T=((T_{i,j}))_{n\times n}, T_{i,j}=P_iTP_j=0,\,\, \mbox{if}\,\,i>j$;
\item[(3)] $T_{i,i}T_{i,j}=T_{i,j}T_{j,j},\; 1 \leq i,j \leq n.$
\end{enumerate}
Let $X$ be an invertible operator,  set $Q_i:=XP_iX^{-1}, 1\leq i \leq n$. We have   
$$\sum\limits_{i=1}^nQ_i=X(\sum\limits_{i=1}^nP_i)X^{-1}=I,$$
$$ Q_iQ_j=X(P_iP_j)X^{-1}=0,\; i\neq j,$$  
$$Q_iXTX^{-1}Q_j=XP_iX^{-1}XTX^{-1}XP_jX^{-1}=XT_{i,j}X^{-1}=0\;\; \mbox{for}\;\;i>j,$$
and 
\begin{eqnarray*}
Q_iXTX^{-1}Q_i Q_iXTX^{-1}Q_j&=&Q_iXTX^{-1}Q_iXTX^{-1}Q_j \\
 &=&XP_iX^{-1}XTX^{-1}XP_iX^{-1}XTX^{-1}XP_jX^{-1} \\
 &=&XP_iTP_iP_iTP_jX^{-1}=XT_{i,i}T_{i,j}X^{-1}\\
 &=&XT_{i,j}T_{j,j}X^{-1}=XP_iTP_jP_jTP_jX^{-1}\\
 &=&Q_iXTX^{-1}Q_jQ_jXTX^{-1}Q_j.
 \end{eqnarray*}
Thus under the decomposition $\mathcal{H}=\mbox{ran}\,Q_1\dotplus \mbox{ran}\,Q_2\dotplus\cdots\dotplus \mbox{ran}\,Q_n,$ $T$ admits the  upper-triangular matrix representation, that is, 
\begin{equation} 
XTX^{-1}=\left ( \begin{smallmatrix}Q_1XTX^{-1}Q_1 & Q_1XTX^{-1}Q_2&  Q_1XTX^{-1}Q_3& \cdots &  Q_1XTX^{-1}Q_n\\
0& Q_2XTX^{-1}Q_2& Q_2XTX^{-1}Q_3&\cdots& Q_2XTX^{-1}Q_n\\
\vdots&\ddots&\ddots&\ddots&\vdots\\
0&\cdots&0& Q_{n-1}XTX^{-1}Q_{n-1}& Q_{n-1}XTX^{-1}Q_{n}\\
0&\cdots&\cdots&0& Q_nXTX^{-1}Q_n
\end{smallmatrix}\right ).
\end{equation}
 
Note that 
\begin{eqnarray*}
Q_iXTX^{-1}Q_j&=&XP_iX^{-1}XTX^{-1}XP_jX^{-1} \\
 &=&XP_iTP_jX^{-1}\\
 &=&XT_{i,j}X^{-1} \\
 &=&X\phi_{i,j}T_{i,i+1}T_{i+1,i+2},\cdots,T_{j-1,j}X^{-1}\\
 &=&X\phi_{i,j}X^{-1}XT_{i,i+1}X^{-1}XT_{i+1,i+2}X^{-1},\cdots,XT_{j-1,j}X^{-1}\\
 &=&X\phi_{i,j}X^{-1}XP_iTP_{i+1}X^{-1}XP_{i+1}TP_{i+2}X^{-1},\cdots,XP_{j-1}TP_jX^{-1}\\
 &=&X\phi_{i,j}X^{-1}Q_iXTX^{-1}Q_{i+1}Q_{i+1}XTX^{-1}Q_{i+2},\cdots,Q_{j-1}XTX^{-1}Q_j\\
 \end{eqnarray*}
 
 and 
 
 \begin{eqnarray*}
 X\phi_{i,j}X^{-1}Q_iXTX^{-1}Q_i&=&X\phi_{i,j}X^{-1}XP_iX^{-1}XTX^{-1}XP_iX^{-1}\\&=&Q_iXTX^{-1}Q_iX\phi_{i,j}X^{-1}.
 \end{eqnarray*}

Finally, we will show the Property $(H)$ remains intact under the similarity transformation for operators in $\mathcal{CFB}_n(\Omega)$. Since $Q_i XTX^{-1}Q_i= X P_iTP_i X^{-1}$ for $1\leq i\leq n$, so by Lemma \ref{SPH}, $Q_i XTX^{-1}Q_i, Q_{i+1} XTX^{-1}Q_{i+1}$ satisfy the Property (H). Hence $XTX^{-1}$ also satisfies the Property (H). Thus $XTX^{-1}$  belongs to $\mathcal{CFB}_n(\Omega)$. Hence  $\mathcal{CFB}_n(\Omega)$ is a similarity invariant set. 

Now, by using the similarity orbit theorem, we prove that $\mathcal{CFB}_n(\Omega)$ is norm dense in $B_n(\Omega)$.  


By Lemma \ref{JWdslemma}, we only need to prove that for any $T\in B_n(\Omega)$ with $\sigma_{lre}(T)=\bar{\Phi}$, $\Phi$ is an analytic Cauchy domain, we can find $T_{\varepsilon}\in \mathcal{CFB}_n(\Omega)$ such that $\|T_{\varepsilon}-T\|<\varepsilon.$ 

Since $\sigma_{lre}(T)=\Phi$, then $\rho_{s-F}(T)$ only has finite many components denoted by $\{\Omega_i, n_i
\}^n_{i=1}$, where $n_i=\mbox{dim ker }(T-w_i),$ for any $w_i\in \Omega_i$, $i=1,2,\cdots, n$. By Lemma  \ref{cdl1},  without loss  of generality, we can assume that  $\Omega_1=\Omega$. Since $\Phi$ is an analytic 
Cauchy domain, then $\Omega$ is an analytic connected Cauchy domain. Let $H_{z_i}(\Omega_i)$ be the multiplication operator on $\mathcal{H}^2(\Omega_i,d\mu_i)$ and 
$$B=H_{z_1}(\Omega_1)\oplus\bigoplus\limits_{i=2}^{n}H^{(n_i)}_{z_i}(\Omega_i)\oplus M$$
where $\mu_i$ is a Lebesgue measure and  $M$ is a diagonal normal operator such that $\sigma(M)=\sigma_{lre}(M)=\bar{\Phi}.$ Applying Lemma \ref{Herlemma} to the operator $B$, there exists a compact operator 
$K$ with $\|K\|<\varepsilon$ such that $B+K\in B_1(\Omega)$. 
Let 
\begin{equation}
T_{\varepsilon}=\left ( \begin{smallmatrix}H_{z_1}(\Omega_1) & I& 0& \cdots & 0&0\\
0&H_{z_1}(\Omega_1)&I&\cdots&0&0\\
\vdots&\ddots&\ddots&\ddots&\vdots&\vdots\\
0&\cdots&0&H_{z_1}(\Omega_1)&I&0\\
0&\cdots&0&0&H_{z_1}(\Omega_1)&0\\
0&\cdots&\cdots&0&0&B+K
\end{smallmatrix}\right ).
\end{equation}
The spectrum pictures of $T_{\varepsilon}$ and $T$ are same. 
Thus, by Lemma \ref{SOT}, there exists invertible operators $\{X_n\}^{\infty}_{n=1}$ such that 
$\lim\limits_{n\rightarrow \infty} X_nT_{\varepsilon}X^{-1}_n=T$. 
Since  $T_{\varepsilon}$ is in $\mathcal{CFB}_n(\Omega)$ and $\mathcal{CFB}_n(\Omega)$ is a similarity invariant set, so  $X_nT_{\varepsilon}X^{-1}_n$ in $\mathcal{CFB}_n(\Omega)$ for all $n$. This finishes the proof of this theorem.

\end{proof}

%
%



\begin{rem} Let $T\in \mathcal{CFB}_n(\Omega)$  and $T=((T_{i,j}))_{n\times n}$ be the $n\times n$ upper-triangular matrix form under a topological direct decomposition of $\mathcal{H}={\mathcal H}_1 \dotplus {\mathcal H}_2\dotplus \cdots  \dotplus{\mathcal H}_n.$ Let $t_n$ be a non zero section of $E_{T_{n,n}}$. Set  $t_{i}:= T_{i,i+1}(t_{i+1}),\; 1\leq i\leq n-1$. It is easy to see that $t_i$ is a section of the vector bundle $E_{T_{i,i}}$. 
We define $\theta_{i,i+1}(T)=\frac{\|T_{i,i+1}(t_{i+1})\|^2}{\|t_{i+1}\|^2}$ and call it generalized second fundamental form.

\end{rem}

\begin{rem} For any topological direct decomposition of $\mathcal{H}$,  ${\mathcal H}={\mathcal H}_1 \dotplus {\mathcal H}_2\dotplus \cdots  \dotplus{\mathcal H}_n$, 
 there exist $n$ idempotents $P_1, P_2,\ldots, P_n$ such that $\sum\limits_{i=1}^nP_i=I, P_iP_j=0,\, i\neq j$ and $\mbox{ran}\,P_i={\mathcal H}_i$.  
Then we can find an invertible operator $X$ such that  $\{Q_i\}^n_{i=1}=\{XP_iX^{-1}\}^n_{i=1}$ be a set of orthogonal projections with  $Q_iQ_j=0,\,i\neq j$. Furthermore, 
$${\mathcal H}=X{\mathcal H}_1\oplus X{\mathcal H}_2 \oplus \cdots  \oplus X{\mathcal H}_n$$
where $X{\mathcal H}_i= \mbox{ran}\, Q_i.$ 
Suppose  $T\in {\mathcal {CFB}}_n(\Omega)$ has the  upper-triangular matrix representation according to a topological direct decomposition of $\mathcal{H}$. By the proof of Proposition \ref{ND}, we see that  $XTX^{-1}\in {\mathcal {CFB}}_n(\Omega)$ according to a orthogonal direct decomposition of $\mathcal{H}$ induced by $X$ above. 
 
From now on we assume that the operators in ${\mathcal {CFB}}_n(\Omega)$ have upper-triangular matrix representation with respect to an orthogonal direct sum decomposition of $\mathcal{H}$. 
\end{rem}

\section{Sufficient conditions for the Property $(H)$}

In this section, we study the ``Property (H)". This property will play an important role in our study on the similarity problem for operators in the class $\mathcal{CFB}_n(\Omega)$. We would like to know under what conditions given  two bounded linear operators in $B_1(\Omega)$ satisfy the Property (H).

Let $T_1, T_2$ be bounded linear operators in $B_1(\Omega)$ and $X$ be a bounded operator such that $T_1 X=X T_2$ and $X=T_1,Y-YT_2$ for some bounded linear operator $Y$. We would like to find a sufficient condition, so that $X$ becomes zero. It is well known that  $T_{i}\sim_{u} (M^*_z, \mathcal{H}_{{K}_{i}}), 1\leq i\leq 2$.

First, we discuss a condition  which ensure that only intertwining operator between $T_1$ and $T_2$ will be the zero operator, that is,  if $T_1X=XT_2$, then $X=0$.  A natural sufficient  condition for this question is 
\begin{equation}\label{condition1}\lim\limits_{w\rightarrow \partial \Omega}\frac{{K}_{1}(w,w)}{{K}_{2}(w,w)}=\infty .\end{equation}
Indeed, when $T_{1}X=XT_{2}$, there exists a holomorphic function $\phi$ defined on $\Omega$ such that $X({K}_2(\cdot,{w}))=\overline{\phi(w)}{K}_1(\cdot, {w}))$. 
By the condition \ref{condition1} and the maximum modulus principle, it follows that $\phi=0$ and hence $X=0$.  For example, consider $S^*_1$ and $S^*_2$ denote the adjoints of 
Hardy shift and Bergman shift, respectively. It is well known there exists no non zero bounded linear operator $X$ such that $S^*_2X=XS^*_1$ (since 
$\lim\limits_{w\rightarrow \partial \mathbb{D}}\frac{(1-|w|^2)^{-2}}{(1-|w|^2)^{-1}}=\lim\limits_{w\rightarrow \partial \mathbb{D}}(1-|w|^2)^{-1}=\infty$).

However, it is not clear what would be a sufficient condition for the 
``Property (H)'' in terms of reproducing kernels as above. Now we will discuss some criterions to decide when given operators $T_1, T_2$ satisy of the Property (H).

\begin{lem}\label{Hal}\cite{Hal} Let $X, T$ be bounded linear operators defined on $\mathcal {H}$. If $X \in \ker\,\delta_{T} \cap \mbox{ran} \,\delta_{T}$, then $\sigma(X)=\{0\}$. 

\end{lem}

\begin{lem}\cite{JJKM}\label{JJKMLemma}
Suppose $T_1$ and $T_2$ are two Cowen-Douglas operators in $B_1(\Omega)$, and $S$ is a bounded operator intertwines $T_1$ and $T_2$, that is, $T_1S = ST_2$. Then $S$ is non-zero if and only if range of $S$ is dense.

\end{lem}

\begin{prop}Let $T_1,T_2$ be bounded linear operators defined on $\mathcal {H}$. If $\ker\, \tau_{T_2,T_1}\neq \{0\}$ and $\{T_2\}^{\prime}$ is semi-simple, then $T_1, T_2$ satisfy  the Property $(H)$.

\end{prop}

\begin{proof} We want to show that $T_1, T_2$ satisfy the Property (H), that is, $\ker\, \tau_{T_1,T_2}\cap \mbox{ran}\, \tau_{T_1,T_2}=0$. Suppose on contrary $\ker\, \tau_{T_1,T_2}\cap \mbox{ran}\, \tau_{T_1,T_2}\neq 0$. Let $X\in \ker\, \tau_{T_1,T_2}\cap \mbox{ran}\, \tau_{T_1,T_2}$ and $X$ is non zero. There exists a bounded operator $Z$ such that $X=T_1Z-ZT_2$  and $T_1X=XT_2$. Since $\ker \,\tau_{T_2,T_1}\neq \{0\}$, there exists a non-zero bounded linear operator $Y$ such that  $YT_1=T_2Y$. We have 
\begin{eqnarray*}
YX&=&YT_1Z-YZT_2\\
&=&T_2YZ-YZT_2
\end{eqnarray*}
and 
\begin{eqnarray*}
YXT_2&=& YT_1X\\
&=&T_2YX.
\end{eqnarray*}  

Thus  $YX\in \ker\,\tau_{T_2}\cap \mbox{ran}\,\tau_{T_2}$. By Lemma \ref{Hal}, it follows that  $\sigma(YX)=0$. Since $X$ is non zero, by Lemma \ref{JJKMLemma}, so the range of $X$ is dense.  Since $\{T_2\}^{\prime}$ is semi-simple and $X\neq 0$,  therefore we have  $Y=0$. This is a contradiction. This completes the proof.
\end{proof}

\begin{prop}\label{JiH} Let $T_1,T_2$ be bounded linear operators on $\mathcal {H}$ and $S_2$ be the right inverse of $T_2$. If $\lim\limits_{n\rightarrow \infty} \frac{\|T^n_1\|  \|S^n_2\|}{n}=0$, then the Property (H) holds.

\end{prop}

\begin{proof} Let  $X, Y$ be linear bounded operators on $\mathcal{H}$ such that $T_1X=XT_2$ and $X=T_1Y-YT_2$. We  claim that  $T^n_1Y-YT^n_2=nT^{n-1}_1X$ for $n\in \mathbb{N}$.  In fact, for $n=1$, the conclusion follows from the assumption.  For $n>1$,  we have 
\begin{eqnarray*}
T^n_1Y-YT^n_2&=&T^n_1Y-T^{n-1}_1YT_2+T^{n-1}_1YT_2-T^{n-2}_1YT^2_2+T^{n-2}_1YT^2_2-\cdots+T_1YT^{n-1}_2-YT^{n}_2\\
&=&T^{n-1}_1(T_1Y-YT_2)+T^{n-2}_1(T_1Y-YT_2)T_2+\cdots+(T_1Y-YT_2)T^{n-1}_2\\
&=&T^{n-1}_1X+T^{n-2}_1XT_2+\cdots+XT^{n-1}_2\\
&=&nT^{n-1}_1X ~~(\mbox{or}~~ nXT^{n-1}_2).
\end{eqnarray*}
Thus we get \begin{eqnarray*}
T^n_1YS^{n}_2-YT^n_2S^n_2&=&nT^{n-1}_1XS^{n}_2\\
&=&nXT^{n-1}_2S^{n}_2\\
&=&nXS_2.
\end{eqnarray*}
Since $T^n_2S^n_2=I$, so we have $$Y=T^n_1YS^n_2-nXS_2,\;\;  n \in \mathbb{N}.$$ 
Therefore, for $n$ in $\mathbb{N}$, we have

\begin{eqnarray}\label{phe1}
\|Y\|&=&\|nXS_2-T^n_1YS^n_2\| \nonumber\\
&\geq &n\|XS_2\|-\|T^n_1YS^n_2\| \nonumber\\
&=&n\big(\|XS_2\|-\frac{\|T^n_1YS^n_2\|}{n}\big).
\end{eqnarray}
If $X=0$, we are done. Suppose $X$ is non zero. Since $Y$ is a bounded linear operator, from equation (\ref{phe1}), it follows that $Y=0$ and hence $X=0$. This is a contraction. This completes the proof.

\end{proof}

\begin{prop}\label{Jiang1}  Let $A,B\in B_1(\mathbb{D})$ be backward shift operators with weighted sequences $\{a_i\}^{\infty}_{i=1}$ and $\{b_i\}^{\infty}_{i=1}$, respectively. 
If $\lim\limits_{n\rightarrow \infty}n\frac{\prod\limits^n_{k=1}b_k}{\prod\limits^n_{k=1}a_k}=\infty$, then the following statements hold:
\begin{enumerate}
\item[(i)]If $X$  intertwines  $A$ and $B$, i.e. $AX=XB$, then there exists an ONB $\{e_i\}^{\infty}_{i=1}$of ${\mathcal H}$  such that  that the matrix form of $X$ with respect to $\{e_i\}^{\infty}_{i=1}$ has the form
$$X=\begin{pmatrix}
x_{1,1} & x_{1,2} & x_{1,3}&\cdots&x_{1,n}&\cdots\\
&x_{2,2}&x_{2,3}&\cdots&x_{2,n} &\cdots\\
&&\ddots&\ddots&\vdots&\cdots\\
&&&x_{n-1,n-1}&x_{n-1,n}&\cdots\\
&&&&x_{n,n}&\ddots\\
&&&&&\ddots\\
\end{pmatrix},$$ where 
$x_{n,n+j}=\frac{\prod\limits_{k=1}^{n-1} b_{k+j}}{\prod \limits_{k=1}^{n-1}a_k}x_{1,1+j}, \;j=0,1,2\ldots, n=1,2,\ldots.$
\item[(ii)] $X\in \ker\, \tau_{A,B} \cap  \mbox{ran}\, \tau_{A,B} $ if and only if $X=0$. 
\end{enumerate}

Furthermore, if we replace $A$ and $B$ by $\phi(A)$ and $\phi(B)$, respectively, where $\phi$ is a univalent analytic function defined on $\overline{\mathbb{D}}$, then above conclusions continue to hold. 

\end{prop}

\begin{proof}
Commuting relation $AX=XB$ forces $X$ to be in upper triangular form. We  consider the following equation 
$$ \begin{pmatrix}
0 & a_1& 0&\cdots&0&\cdots\\
&0&a_2&\cdots&0 &\cdots\\
&&\ddots&\ddots&\vdots&\cdots\\
&&&0&a_{n-1}&\cdots\\
&&&&0&a_n\\
&&&&&\ddots\\
\end{pmatrix}\begin{pmatrix}
x_{1,1} & x_{1,2} & x_{1,3}&\cdots&x_{1,n}&\cdots\\
&x_{2,2}&x_{2,3}&\cdots&x_{2,n} &\cdots\\
&&\ddots&\ddots&\vdots&\cdots\\
&&&x_{n-1,n-1}&x_{n-1,n}&\cdots\\
&&&&x_{n,n}&\ddots\\
&&&&&\ddots\\
\end{pmatrix}$$
$$=\begin{pmatrix}
x_{1,1} & x_{1,2} & x_{1,3}&\cdots&x_{1,n}&\cdots\\
&x_{2,2}&x_{2,3}&\cdots&x_{2,n} &\cdots\\
&&\ddots&\ddots&\vdots&\cdots\\
&&&x_{n-1,n-1}&x_{n-1,n}&\cdots\\
&&&&x_{n,n}&\ddots\\
&&&&&\ddots\\
\end{pmatrix}\begin{pmatrix}
0 & b_1& 0&\cdots&0&\cdots\\
&0&b_2&\cdots&0 &\cdots\\
&&\ddots&\ddots&\vdots&\cdots\\
&&&0&b_{n-1}&\cdots\\
&&&&0&b_n\\
&&&&&\ddots\\
\end{pmatrix},$$
by comparing the elements in $(i,j)$ position and after a simple calculation, we will get the statement $(\mbox{i})$, that is, 
\begin{equation}\label{3.4}x_{n,n+j}=\frac{\prod\limits_{k=1}^{n-1} b_{k+j}}{\prod \limits_{k=1}^{n-1}a_k}x_{1,1+j}, \;\; j=0,1,2\ldots, n=1,2,\ldots.\end{equation}
Thus we only need to prove statement $(ii)$. Notice that $Ae_1=Be_1=0, Xe_1=x_{1,1}e_1$ and 
$$AYe_1-YBe_1=Xe_1=x_{1,1}e_1.$$ 
Since $AYe_1=x_{1,1}e_1$, so there exist $\alpha^1_1, \alpha^1_2$ in $\mathbb{C}$ such that $Ye_1=\alpha^1_1e_1+\alpha^1_2e_2.$  From  
$$AYe_2-YBe_2=Xe_2=x_{1,2}e_1+x_{2,2}e_2,$$ it follows that  
\begin{eqnarray*}AYe_2&=&b_1\alpha^1_1e_1+b_1\alpha^1_2e_2+x_{1,2}e_1+x_{2,2}e_2\\
&=&(b_{1}\alpha^1_1+x_{1,2})e_1+(b_1\alpha^1_2+x_{2,2})e_2.
\end{eqnarray*}
Similarly, we can find $\alpha^2_1,\alpha^2_2, \alpha^2_3 \in \mathbb{C}$ such that 
$$Ye_2=\alpha^2_1e_1+\alpha^2_2e_2+\alpha^2_3e_3.$$
Inductively we see that for any $n>0$, $$Ye_n\in \bigvee\{e_1,e_2,\cdots, e_{n+1}\}.$$ It follows that the matrix form of $Y$ according to $
\{e_i\}^{\infty}_{i=1}$ is as follows :
$$ Y=\begin{pmatrix}
y_{1,1} & y_{1,2} & y_{1,3}&\cdots&y_{1,n}&\cdots&\cdots\\
y_{2,1}&y_{2,2}&y_{2,3}&\cdots&y_{2,n} &\cdots&\cdots\\
&y_{3,2}&y_{3,3}&\cdots&y_{3,n}&\cdots&\cdots\\
&&\ddots&\ddots&\ddots&\ddots&\ddots\\
&&&y_{n-1,n-2}&y_{n-1,n-1}&y_{n-1,n}&\cdots\\
&&&&y_{n,n-1}&y_{n,n}&\ddots\\
&&&&&\ddots&\ddots\\
\end{pmatrix}.$$

From the equation $AY-YB=X$, it is easy to see that
$$y_{n,n-1}=n\frac{\prod^n_{k=1}b_k}{a_{n+1}\prod^n_{k=1}a_k}x_{1,1}.$$
By the assumption of the lemma, $$\lim\limits_{n\rightarrow \infty}n\frac{\prod\limits^n_{k=1}b_k}{\prod\limits^n_{k=1}a_k}=\infty.$$
So we have $x_{1,1}=0$ and $y_{n,n-1}=0.$ By the equation (\ref{3.4}), we get $x_{n,n}=0, \;\;n=1,2,\ldots$. 

Now assume that $x_{1,2}\neq 0$, then it follows that 
$$y_{n,n}=n\frac{\prod\limits^n_{k=1}b_{k+1}}{a_{n+1}\prod\limits^n_{k=1}a_k}x_{1,2}+\frac{\prod\limits^n_{k=1}b_k}{a_{n+1}\prod\limits^{n}_{k=1}a_k}y_{1,1}.$$
Since $A,B\in B_1(\mathbb{D})$, there exist $d, M_0\in \mathbb{R}^+$ such that 
$$\min\limits_k\{|a_k|,|b_k|\}>d,\;\;\;\; \max\limits_k\{|a_k|,|b_k|\}<M_0.$$  
We have 
$$n\frac{\prod\limits^n_{k=1}b_{k+1}}{a_{n+1}\prod\limits^n_{k=1}a_k}\geq n\frac{b_{n+1}}{a_{n+1}b_1}\frac{\prod\limits^n_{k=1}b_{k}}{\prod\limits^n_{k=1}a_k}\geq \frac{d}{M_0b_1}n\frac{\prod\limits^n_{k=1}b_{k}}{\prod\limits^n_{k=1}a_k}\rightarrow \infty.$$ 
Notice that $\frac{\prod\limits^n_{k=1}b_k}{a_{n+1}\prod\limits^{n}_{k=1}a_k}y_{1,1}$ is bounded. If $x_{1,2}\neq 0$, then $y_{n,n}\rightarrow \infty$ as $n\rightarrow \infty$. 
Thus we have $x_{1,2}=0$.  Using equation (\ref{3.4}) again, we get $x_{n,n+1}=0$ for any $n>1.$ 

Set $E_0:=\mbox{diag}\{y_{1,1},y_{2,2}\cdots,y_{n,n}\cdots\}$, by a directly calculation, it is easy to see that $AE_0=E_0B$ and hence 
$$A(Y-E_0)-(Y-E_0)B=X.$$ So for sake of simplicity, we continue to denote $Y-E_0$ by $Y$. In this case, we also have 
$$y_{n,n+1}=n\frac{\prod\limits^n_{k=1}b_{k+2}}{\prod\limits^n_{k=1}a_k}x_{1,3}+\frac{\prod\limits^{n}_{k=1}b_{k+1}}{a_{n+1}\prod\limits^n_{k=1}a_k}y_{1,2}.$$
By a similar argument as above, we have $x_{1,3}=0$. Set $$E_1=\begin{pmatrix}
0 & y_{1,2}& 0&\cdots&0&\cdots\\
&0&y_{2,3}&\cdots&0 &\cdots\\
&&\ddots&\ddots&\vdots&\cdots\\
&&&0&y_{n-1,n}&\cdots\\
&&&&0&y_{n,n+1}\\
&&&&&\ddots\\
\end{pmatrix}.$$

Thus we have  $AE_1=E_1B$ and $A(Y-E_1)+(Y-E_1)B=X$. We again continue to denote $Y-E_1$ by $Y$. By repeating the above process, we see that for any $j$, $x_{n,n+j}=0$. Thus $X=0$. This finishes the proof of statement (ii). 

At last, we will show the conclusion above continue to hold for $\phi(A)$ and $\phi(B)$.  Without of loss generality, we can also assume that $\phi(z)=\sum\limits_{n=1}^{\infty}k_nz^n$. Suppose that there exists a bounded operator $X$ such that 
$\phi(A)X=X\phi(B)$. By a similar argument, we have  
$$x_{n,n}=\frac{\prod\limits_{k=1}^{n-1} b_{k}}{\prod \limits_{k=1}^{n-1}a_k}x_{1,1},\;\; n=1,2\ldots.$$
Suppose that $\phi(A)Y-Y\phi(B)=X$. By a directly calculation, we see that 
$$y_{n+1,n}=n\frac{\prod\limits^{n-1}_{k=1}b_k}{\prod\limits^{n-1}_{k=1}a_k}x_{1,1},\;\;n=1,2,\ldots.$$
Thus we have $
x_{1,1}=0$ and $y_{n,n-1}=0.$ By the equation \ref{3.4}, we have $x_{n,n}=0, n=1,2,\ldots$. Set $E_0=\mbox{diag}\{y_{1,1},y_{2,2},\ldots,y_{n,n},\ldots\}$, by a directly calculation, we have $AE_0=E_0B$ and hence $\phi(A)E_0=E_0\phi(B)$ and 
$$\phi(A)(Y-E_0)-(Y-E_0)\phi(B)=X.$$ So for sake of simplicity, we still use $Y$ to denote $Y-E_0$. Now repeating the proof of the statement $(ii)$, it can also be shown that  $X$ is equal to the zero operator.

\end{proof}

\begin{cor}\label{HMPH}Let $M_{i,z}$ be the multiplication operator on the reproducing kernel Hilbert space $H_{K_i}$, where $K_i(z,w)=\frac{1}{(1-z\bar{w})^{\lambda_i}}$, $z,w\in \mathbb{D}, \,1\leq i\leq 2$. If $\lambda_2-\lambda_1<2$, then $M^*_{1,z}$ and $M^*_{2,z}$ satisfy  the Property $(H)$. 
\end{cor}

\begin{proof}
Let $a_n(\lambda_i)$ denote the coefficient of $\bar{w}^nz^n$ in the power series  expansion for $K_i(z,w)$, $i=1,2$.  Then $M^*_{i,z}$ is a backward weight shift with 
$w^{(\lambda_i)}_n=\frac{\sqrt{a_n(\lambda_i)}}{\sqrt{a_{n+1}(\lambda_i)}}, i=1,2$.  By Stirling's formula (see more details in 3.3.1 \cite{JJM}), we have that 
$$\prod\limits^n_{k=0}w^{(\lambda_i)}_k=\frac{\sqrt{a_0(\lambda_i)}}{\sqrt{a_{n+1}(\lambda_i)}}\sim (n+1)^{\frac{1-\lambda_i}{2}}, i=1,2.$$
Then we have that 
$$\frac{\prod\limits^n_{k=0}w^{(\lambda_2)}_k}{\prod\limits^n_{k=0}w^{(\lambda_1)}_k}
\sim (n+1)^{-\frac{\lambda_2-\lambda_1}{2}}.$$
If  $ \lambda_2-\lambda_1<2$, then $\lim\limits_{n\rightarrow \infty}
n\frac{\prod\limits^n_{k=0}w^{(\lambda_2)}_k}{\prod\limits^n_{k=0}w^{(\lambda_1)}_k}=\infty$.
  Thus, by Proposition \ref{Jiang1},  $M^*_{1,z}$ and $M^*_{2,z}$ satisfy  the Property $(H)$. 
\end{proof}

\section{Similarity of Operators in $\mathcal{CFB}_n(\Omega)$}

In this section, we give complete similarity invariants for operators in $\mathcal{CFB}_n(\Omega)$ which involve the curvature and the second fundamental form. This is quiet different from the case of quasi-homogenous operator class (See Theorem \ref{Main1}). To prove main theorem of this section, we need the following concepts and  lemmas.

An operator $T$ in ${\mathcal L}(\mathcal H)$ is said to be strongly
irreducible, if there is no non-trivial idempotent operator in
$\{T\}^{\prime}$, where $\{T\}^{\prime}$ denotes
the commutant of $T$, i.e., $\{T\}^{\prime}=\{B{\in}{\mathcal L}({\mathcal H}): {TB=BT}\}$. It
can be proved that for any $T\in B_1(\Omega)$, $T$ is strongly
irreducible. We denote the set of all the strongly
irreducible operators by the symbol ``($SI$)''.

An operator $T$ in ${\mathcal L}({\mathcal H})$ is said to have
finite strongly irreducible decomposition, if there exist idempotents $P_1,
P_2, \ldots, P_n$  in $\{T\}^{\prime}$ such that

$1.\,\, P_iP_j={\delta}_{ij}P_i$  for $1{\leq}i,j{\leq}n<+{\infty},$
where ${\delta}_{ij}=\left\{
\begin{array}{cc}
0,&i{\not=}j\\
1,&i=j
\end{array}
\right. $;

$2. \,\, \sum\limits_{i=1}^{n}P_i=I_{\mathcal H},$ where
$I_{\mathcal H}$ denotes the identity operator on $\mathcal H$;

$3. \,\, T|_{P_i{\mathcal H}}$ is strongly irreducible for $i=1, 2, \ldots, n.$

Every Cowen-Douglas operator can be written as the direct sum of finitely
many strongly irreducible Cowen-Douglas operators (see \cite{JW}, chapter 3).
We call $P=(P_1, P_2, \ldots, P_n)$ a unit finite strongly
irreducible decomposition of $T$. 
Let $T$ has a finite $(SI) $ decomposition and $P=\{P_i \}^n_{i=1}$ and $Q=\{Q_i\}^{m}_{i=1}$ be two unit finite $(SI)$ decompositions of T. 
We say that $T$ has unique
strongly irreducible decomposition up to similarity if the following
conditions are satisfied:

1. $m=n;$  and

2. There exists an invertible operator $X$ in $\{T\}^{\prime}$ and a permutation ${\Pi}$ of the set $\{1,2,\cdots,n\}$ such that
$XQ_{{\Pi}(i)}X^{-1}=P_i$ for $1{\leq}i{\leq}n.$

\begin{lem}[Theorem 5.5.12, \cite{JW1}]\label{l1} Let $T$ be a Cowen-Douglas operator in $B_n(\Omega)$. The operator $T$ has a unique $(SI)$ decomposition. 

\end{lem}

\begin{lem}[Theorem 5.5.13, \cite{JW1}]\label{l2} Let $T=\bigoplus\limits_{i=1}^kT^{(n_i)}, \tilde{T}=\bigoplus\limits_{j=1}^s\tilde{T}^{(m_j)}$ be two Cowen-Douglas operators, where 
$T_i, \tilde{T}_j\in (SI) $ for  any $i, j$ and $T_i\not\sim_s T_{i^{\prime}}, \tilde{T}_j\not\sim_s \tilde{T}_{j^{\prime}}.$  Then $T\sim_s\tilde{T}$ if and only if 
$k=s$ and there exists a permutation ${\Pi}$ such that $T_i\sim_s \tilde{T}_{\Pi(i)}$ and $n_i=m_{\Pi(i)}, i=1,2,\cdots, k$. 

\end{lem}

By lemma \ref{l1}, and lemma \ref{l2}, we only need to consider when two strongly irreducible operators in $\mathcal{CFB}_n(\Omega)$ are similar equivalent. 
The  similarity classification for general case will follows by  lemma \ref{l2}.  Thus, in the following, we will assume $T\in \mathcal{CFB}_n(\Omega)$ is strongly irreducible operator.

\begin{lem} \label{SI}Let $T\in \mathcal{CFB}_n(\Omega)$. Then $T$ is strongly irreducible  if and only if  $T_{i,i+1}\neq 0$ for any $i=1,2\cdots, n-1$.

\end{lem}
\begin{proof} Let $T$ be a strongly irreducible operator in $\mathcal{CFB}_n(\Omega)$.  Suppose on contrary that $T_{k-1,k}=0$ for some $k$.  For i, j with $i+1\leq k\leq j$, we have  
\begin{eqnarray*}T_{i,j}&=&\phi_{i,j}T_{i,i+1}T_{i+1,i+2}\cdots T_{k-1,k}\cdots T_{j-1,j}\\
&=& 0.
\end{eqnarray*}
Thus $T$ has the following matrix form:
\begin{equation}
T=\left ( \begin{smallmatrix}T_{1,1} & T_{1,2}& \cdots&T_{1,k-1}&0&0& \cdots & 0\\
0&T_{2,2}&\cdots&T_{2,k-1}&0&0&\cdots&0\\
\vdots&\ddots&\ddots&\vdots&\vdots&\vdots&\vdots&\vdots\\
0&0&\cdots&T_{k-1,k-1}&0&\cdots&0&0\\
0&\cdots&0&0&T_{k,k}&T_{k,k+1}&\cdots&T_{k,n}\\
0&\cdots&\cdots&0&0&T_{k+1,k+1}&\cdots&T_{k+1,n}\\
0&0&\cdots&\cdots&\vdots&\ddots&\ddots&\vdots\\
0&0&\cdots&\cdots&0&0&\cdots&T_{n,n}\\
\end{smallmatrix}\right ).
\end{equation}
So $T$ is strongly reducible. This is a contradiction to the fact that $T$ is strongly irreducible. This finishes the proof of necessary part.   

For the sufficient part,  suppose that  each $T_{i,i+1}$ is non-zero operator. By the Definition  \ref{CFB1},  $T_{i,i}$ and $T_{i+1,i+1}$ satisfy the 
Property $(H)$.  Since $T_{i,i}T_{i,i+1}=T_{i,i+1}T_{i+1,i+1}$, it follows that $T_{i,i+1}\not\in ran\tau_{T_{i,i}T_{i+1,i+1}}.$ By a same proof of Proposition 2.22 in \cite{JJKM}, 
it is easy to see that $T$ is strongly irreducible. 
\end{proof}

\begin{lem}\label{J21}
Let $T=((T_{i,j}))_{n\times n}$, $\tilde{T}=((\tilde{T}_{i,j}))_{n\times n}$ be operators in $\mathcal{CFB}_n(\Omega)$. If $T_{i,i}=\tilde{T}_{i,i},$ $T_{i,i+1}=\tilde{T}_{i,i+1}$, then there exists 
a bounded operator $K$ such that $X=I+K$ is invertible and $XT=\tilde{T}X.$ 
\end{lem}

\begin{proof} To find $K$, we need to solve the equation 
\begin{equation}\label{mle1}
(I+K)T=\tilde{T}(I+K)
\end{equation} Set $X:=I+K$, where 
$$K=\left ( \begin{smallmatrix}0 & K_{1,2}& K_{1,3}& \cdots & K_{1,n}\\
0&0&K_{2,3}&\cdots&K_{2,n}\\
\vdots&\ddots&\ddots&\ddots&\vdots\\
0&\cdots&0&0&K_{n-1,n}\\
0&\cdots&\cdots&0&0
\end{smallmatrix}\right ).$$
From equation (\ref{mle1}), we have
\begin{eqnarray}\label{4.1}
&&\left ( \begin{smallmatrix}1 & K_{1,2}& K_{1,3}& \cdots & K_{1,n}\\
0&1&K_{2,3}&\cdots&K_{2,n}\\
\vdots&\ddots&\ddots&\ddots&\vdots\\
0&\cdots&0&1&K_{n-1,n}\\
0&\cdots&\cdots&0&1
\end{smallmatrix}\right )\left ( \begin{smallmatrix}T_{1,1} & T_{1,2}& {T}_{1,3}& \cdots & {T}_{1,n}\\
0&T_{2,2}&T_{2,3}&\cdots& {T}_{2,n}\\
\vdots&\ddots&\ddots&\ddots&\vdots\\
0&\cdots&0&T_{n-1,n-1}&T_{n-1,n}\\
0&\cdots&\cdots&0&T_{n,n}
\end{smallmatrix}\right ) \nonumber\\
&=&\left ( \begin{smallmatrix}T_{1,1} & T_{1,2}& \tilde{T}_{1,3}& \cdots & \tilde{T}_{1,n}\\
0&T_{2,2}&T_{2,3}&\cdots&\tilde{T}_{2,n}\\
\vdots&\ddots&\ddots&\ddots&\vdots\\
0&\cdots&0&T_{n-1,n-1}&T_{n-1,n}\\
0&\cdots&\cdots&0&T_{n,n}
\end{smallmatrix}\right )\left ( \begin{smallmatrix}1 & K_{1,2}& K_{1,3}& \cdots & K_{1,n}\\
0&1&K_{2,3}&\cdots&K_{2,n}\\
\vdots&\ddots&\ddots&\ddots&\vdots\\
0&\cdots&0&1&K_{n-1,n}\\
0&\cdots&\cdots&0&1
\end{smallmatrix}\right ).
\end{eqnarray}
To find $K_{i,j}$, we follow following steps.

\noindent {\bf Step 1}: 
For $1\leq i\leq n-1$, by equating the $(i,i+1)$th entry of equation (\ref{4.1}), we have $T_{i,i+1}+ K_{i,i+1}T_{i+1,i+1}= T_{i,i} K_{i,i+1}+ T_{i,i+1}$, that is, $ K_{i,i+1}T_{i+1,i+1}= T_{i,i} K_{i,i+1}$.
For $1\leq i \leq n-2$, by comparing $(i,i+2)$th entry of equation (\ref{4.1}), we have
\begin{eqnarray}\label{l4e1}T_{i,i+2}+K_{i,i+1}T_{i+1,i+2}+K_{i,i+2}T_{i+2,i+2}=T_{i,i}K_{i,i+2}+T_{i,i+1}K_{i+1,i+2}+\tilde{T}_{i,i+2}.\end{eqnarray}
If $T_{i,i}K_{i,i+2}= K_{i,i+2}T_{i+2,i+2}, \;1\leq i\leq n-2$, then from equation  (\ref{l4e1}) we get
\begin{eqnarray}\label{l4e2}T_{i,i+2}+K_{i,i+1}T_{i+1,i+2}=T_{i,i+1}K_{i+1,i+2}+\tilde{T}_{i,i+2}.\end{eqnarray}
Choose $K_{n-1,n}$ such that $ K_{n-1,n}T_{n,n}= T_{n-1,n-1} K_{n-1,n}$. For $1\leq i\leq n-2$, from equation (\ref{l4e2}), we get $K_{i,i+1}$ which satisfies $ K_{i,i+1}T_{i+1,i+1}= T_{i,i} K_{i,i+1}$.

\noindent {\bf Step 2}:
We compare $(i,i+3)$th entry of equation (\ref{4.1}), we get
 \begin{eqnarray}\label{l4e3}
 &&T_{i,i+3}+ K_{i,i+1}T_{i+1,i+3}+K_{i,i+2}T_{i+2,i+3}+K_{i,i+3}T_{i+3,i+3}\nonumber\\&&=T_{i,i}K_{i,i+3}+T_{i,i+1}K_{i+1,i+3}+\tilde{T}_{i,i+2}K_{i+2,i+3}+\tilde{T}_{i,i+3}.
 \end{eqnarray}
 If $T_{i,i}K_{i,i+3}=K_{i,i+3}T_{i+3,i+3}, 1\leq i\leq n-3$, then from equation (\ref{l4e3}) we have  
 \begin{eqnarray}\label{l4e4}
 &&T_{i,i+3}+ K_{i,i+1}T_{i+1,i+3}+K_{i,i+2}T_{i+2,i+3}\nonumber\\&&=T_{i,i+1}K_{i+1,i+3}+\tilde{T}_{i,i+2}K_{i+2,i+3}+\tilde{T}_{i,i+3}.
 \end{eqnarray}
Choose $K_{n-2,n}$ such that $K_{n-2,n} T_{n,n}=T_{n-2,n-2}K_{n-2,n}$. For $1\leq i\leq n-3$, from equation (\ref{l4e4}), we get $K_{i,i+2}$ which satisfies $T_{i,i}K_{i,i+2}= K_{i,i+2}T_{i+2,i+2}$.

\noindent {\bf Step 3}: 
By following previous steps, suppose we have solved $K_{i,i+l}$ for $ 1\leq i\leq n-l, 1\leq l\leq j-2$.


By comparing the $(i,i+j)$th entry of the equation (\ref{4.1}), we have 
\begin{eqnarray}\label{l4e5}
&&T_{i,i+j}+K_{i,i+1}T_{i+1,i+j}+K_{i,i+2}T_{i+2,i+j}+\cdots+K_{i,i+j}T_{i+j,i+j}\nonumber\\
&=&T_{i,i}K_{i,i+j}+T_{i,i+1}K_{i+1,i+j}+\tilde{T}_{i,i+2} K_{i+2,i+j}+\cdots+\tilde{T}_{i,i+j-1}K_{i+j-1,i+j}+\tilde{T}_{i,i+j}.
\end{eqnarray}
If $T_{i,i+j}K_{i,i+j}=K_{i,i+j}T_{i+j,i+j}, 1\leq i\leq n-j$, then from equation (\ref{l4e5}) we get

\begin{eqnarray}\label{l4e6}
&&T_{i,i+j}+K_{i,i+1}T_{i+1,i+j}+K_{i,i+2}T_{i+2,i+j}+\cdots+K_{i,i+j-1}T_{i+j-1,i+j}\nonumber\\
&=&T_{i,i+1}K_{i+1,i+j}+\tilde{T}_{i,i+2} K_{i+2,i+j}+\cdots+\tilde{T}_{i,i+j-1}K_{i+j-1,i+j}+\tilde{T}_{i,i+j}.
\end{eqnarray}

Choose $K_{n-j+1,n}$ such that $K_{n-j+1,n} T_{n,n}=T_{n-j+1,n-j+1}K_{n-j+1,n}$. For $1\leq i\leq n-j$, from equation (\ref{l4e6}), we get $K_{i,i+j-1}$ which satisfies $T_{i,i}K_{i,i+j-1}= K_{i,i+j-1}T_{i+j-1,i+j-1}$.

\end{proof}

We recall a result from \cite{JJKM} which gives a description of an invertible operator intertwining any two operators in $\mathcal{FB}_n(\Omega)$.

 \begin{prop}\label{pjjkm}
 If $X$ is an invertible operator intertwining two operators $T$ and $\tilde{T}$ from $\mathcal{FB}_n(\Omega)$, then $X$ and $X^{-1}$ are upper triangular.
 \end{prop}

\begin{lem}\label{J3}
Let  $T$ and $\tilde{T}$ be operators  in $\mathcal{CFB}_n(\Omega)$. Let  $X$ be a bounded linear operator of the form
$$X=\left ( \begin{smallmatrix}X_{1,1} & X_{1,2}& X_{1,3}& \cdots & X_{1,n}\\
0&X_{2,2}&X_{2,3}&\cdots&X_{2,n}\\
\vdots&\ddots&\ddots&\ddots&\vdots\\
0&\cdots&0&X_{n-1,n-1}&X_{n-1,n}\\
0&\cdots&\cdots&0&X_{n,n}
\end{smallmatrix}\right ).$$
If $X\tilde{T}=TX$ and X is invertible, then 
\begin{eqnarray*}
&&\left ( \begin{smallmatrix}X_{1,1} & 0& 0& \cdots & 0\\
0&X_{2,2}&0&\cdots&0\\
\vdots&\ddots&\ddots&\ddots&\vdots\\
0&\cdots&0&X_{n-1,n-1}&0\\
0&\cdots&\cdots&0&X_{n,n}
\end{smallmatrix}\right )\left ( \begin{smallmatrix}\tilde{T}_{1,1} & \tilde{T}_{1,2}&0& \cdots & 0\\
0&\tilde{T}_{2,2}&\tilde{T}_{2,3}&\cdots&0\\
\vdots&\ddots&\ddots&\ddots&\vdots\\
0&\cdots&0&\tilde{T}_{n-1,n-1}&T_{n-1,n}\\
0&\cdots&\cdots&0&\tilde{T}_{n,n}
\end{smallmatrix}\right )\\
&=&\left ( \begin{smallmatrix}T_{1,1} & T_{1,2}& 0& \cdots & 0\\
0&T_{2,2}&T_{2,3}&\cdots&0\\
\vdots&\ddots&\ddots&\ddots&\vdots\\
0&\cdots&0&T_{n-1,n-1}&T_{n-1,n}\\
0&\cdots&\cdots&0&T_{n,n}
\end{smallmatrix}\right )\left ( \begin{smallmatrix}X_{1,1} & 0& 0& \cdots & 0\\
0&X_{2,2}&0&\cdots&0\\
\vdots&\ddots&\ddots&\ddots&\vdots\\
0&\cdots&0&X_{n-1,n-1}&0\\
0&\cdots&\cdots&0&X_{n,n}
\end{smallmatrix}\right )
\end{eqnarray*}


\end{lem}

\begin{proof}
By equating the entries of $X\tilde{T}=TX$, we get 

$$X_{i,i}\tilde{T}_{i,i}=T_{i,i}X_{i,i}, ~~ ~~1\leq i\leq n.$$
 Set $Y:=X\,\mbox{diag}\,\{X^{-1}_{1,1},X^{-1}_{2,2},\cdots,X^{-1}_{n,n}\}$, and it is  easy to see that 
\begin{eqnarray*}
Y=\left ( \begin{smallmatrix}I &X_{1,2}X^{-1}_{2,2}& X_{1,3}X^{-1}_{3,3}& \cdots & X_{1,n}X^{-1}_{n,n}\\
0&I&X_{2,3}X^{-1}_{3,3}&\cdots&X_{2,n}X^{-1}_{n,n}\\
\vdots&\ddots&\ddots&\ddots&\vdots\\
0&\cdots&0&I&X_{n-1,n}X^{-1}_{n,n}\\
0&\cdots&\cdots&0&I
\end{smallmatrix}\right ).
\end{eqnarray*}

From $T=X\tilde{T}X^{-1}$, we get $TY= Y\,\mbox{diag}\,\{X_{1,1}\cdots X_{n,n}\}\tilde{T}\,\mbox{diag}\,\{ X_{1,1}\cdots X_{n,n}\}^{-1}$ which is equivalent to   
 
\begin{eqnarray}\label{ml2e1}
&&\left ( \begin{smallmatrix}T_{1,1} & T_{1,2}& {T}_{1,3}& \cdots & {T}_{1,n}\\
0&T_{2,2}&T_{2,3}&\cdots&{T}_{2,n}\\
\vdots&\ddots&\ddots&\ddots&\vdots\\
0&\cdots&0&T_{n-1,n-1}&T_{n-1,n}\\
0&\cdots&\cdots&0&T_{n,n}
\end{smallmatrix}\right )\left ( \begin{smallmatrix}I &X_{1,2}X^{-1}_{2,2}& X_{1,3}X^{-1}_{3,3}& \cdots & X_{1,n}X^{-1}_{n,n}\\
0&I&X_{2,3}X^{-1}_{3,3}&\cdots&X_{2,n}X^{-1}_{n,n}\\
\vdots&\ddots&\ddots&\ddots&\vdots\\
0&\cdots&0&I&X_{n-1,n}X^{-1}_{n,n}\\
0&\cdots&\cdots&0&I
\end{smallmatrix}\right )\nonumber\\
&=&\left ( \begin{smallmatrix}I &X_{1,2}X^{-1}_{2,2}& X_{1,3}X^{-1}_{3,3}& \cdots & X_{1,n}X^{-1}_{n,n}\\
0&I&X_{2,3}X^{-1}_{3,3}&\cdots&X_{2,n}X^{-1}_{n,n}\\
\vdots&\ddots&\ddots&\ddots&\vdots\\
0&\cdots&0&I&X_{n-1,n}X^{-1}_{n,n}\\
0&\cdots&\cdots&0&I
\end{smallmatrix}\right )
\left ( \begin{smallmatrix}T_{1,1}&X_{1,1}\tilde{T}_{1,2}X^{-1}_{2,2}&\cdots& \cdots&X_{1,1}\tilde{T}_{1,n}X^{-1}_{n,n}\\
0&T_{2,2}&X_{2,2}\tilde{T}_{2,3}X^{-1}_{3,3}&\cdots&X_{2,2}\tilde{T}_{2,n}X^{-1}_{n,n} \\
\vdots&\ddots&\ddots&\ddots&\vdots\\
0&0&\cdots&0&T_{n,n}
\end{smallmatrix}\right ).
\end{eqnarray}

For $1\leq i\leq n-1$, from equation (\ref{ml2e1}), we get
 
\begin{eqnarray*}
\left(\begin{matrix}T_{i,i}&T_{i,i+1}\\
0&T_{i+1,i+1}\\ \end{matrix}\right)\left(\begin{matrix}I&X_{i,i+1}X^{-1}_{i+1,i+1}\\
0&I\\ \end{matrix}\right)
= \left(\begin{matrix}I&X_{i,i+1}X^{-1}_{i+1,i+1}\\
0&I\\ \end{matrix}\right)\left(\begin{matrix}T_{i,i}&X_{i,i}\tilde{T}_{i,i+1}X^{-1}_{i+1,i+1}\\
0&T_{i+1,i+1}\\ \end{matrix}\right)\\
\end{eqnarray*}

which is equivalent to
\begin{eqnarray*}\label{4.5}
T_{i,i+1}-X_{i,i}\tilde{T}_{i,i+1}X^{-1}_{i+1,i+1}
=X_{i,i+1}X^{-1}_{i+1,i+1}T_{i+1,i+1}-T_{i,i}X_{i,i+1}X^{-1}_{i+1,i+1}.
\end{eqnarray*}
Consider  
\begin{eqnarray*}
T_{i,i}(T_{i,i+1}-X_{i,i}\tilde{T}_{i,i+1}X^{-1}_{i+1,i+1})
&=&T_{i,i+1}T_{i+1,i+1}-X_{i,i}\tilde{T}_{i,i}\tilde{T}_{i,i+1}X^{-1}_{i+1,i+1}\\
&=&T_{i,i+1}T_{i+1,i+1}-X_{i,i}\tilde{T}_{i,i+1}\tilde{T}_{i+1,i+1}X^{-1}_{i+1,i+1}\\
&=&T_{i,i+1}T_{i+1,i+1}-X_{i,i}\tilde{T}_{i,i+1}X^{-1}_{i+1,i+1}T_{i+1,i+1}\\
&=& (T_{i,i+1}-X_{i,i}\tilde{T}_{i,i+1}X^{-1}_{i+1,i+1}) T_{i+1,i+1}.
\end{eqnarray*}

In other words, $(T_{i,i+1}-X_{i,i}\tilde{T}_{i,i+1}X^{-1}_{i+1,i+1})$ belongs to $\ker\,(\tau_{T_{i,i},T_{i+1,i+1}})\cap \mbox{ran}\,(\tau_{T_{i,i},T_{i+1,i+1}})$. Since $T$ satisfies the Property $(H)$, so we have $$T_{i,i+1}=X_{i,i}\tilde{T}_{i,i+1}X^{-1}_{i+1,i+1},\,\, 1\leq i\leq n-1.$$

Hence
\begin{eqnarray*}
&&\left ( \begin{smallmatrix}X_{1,1} & 0& 0& \cdots & 0\\
0&X_{2,2}&0&\cdots&0\\
\vdots&\ddots&\ddots&\ddots&\vdots\\
0&\cdots&0&X_{n-1,n-1}&0\\
0&\cdots&\cdots&0&X_{n,n}
\end{smallmatrix}\right )\left ( \begin{smallmatrix}\tilde{T}_{1,1} & \tilde{T}_{1,2}&0& \cdots & 0\\
0&\tilde{T}_{2,2}&\tilde{T}_{2,3}&\cdots&0\\
\vdots&\ddots&\ddots&\ddots&\vdots\\
0&\cdots&0&\tilde{T}_{n-1,n-1}&0\\
0&\cdots&\cdots&0&\tilde{T}_{n,n}
\end{smallmatrix}\right )\\
&=&\left ( \begin{smallmatrix}T_{1,1} & T_{1,2}& 0& \cdots & 0\\
0&T_{2,2}&T_{2,3}&\cdots&0\\
\vdots&\ddots&\ddots&\ddots&\vdots\\
0&\cdots&0&T_{n-1,n-1}&0\\
0&\cdots&\cdots&0&T_{n,n}
\end{smallmatrix}\right )\left ( \begin{smallmatrix}X_{1,1} & 0& 0& \cdots & 0\\
0&X_{2,2}&0&\cdots&0\\
\vdots&\ddots&\ddots&\ddots&\vdots\\
0&\cdots&0&X_{n-1,n-1}&0\\
0&\cdots&\cdots&0&X_{n,n}
\end{smallmatrix}\right ).
\end{eqnarray*}

\end{proof}

\begin{cor}
Let $T=((T_{i,j}))_{n\times n}$ and $\tilde{T}=((\tilde{T}_{i,j}))_{n\times n}$ be any two operators in $\mathcal{CFB}_n(\Omega)$. Suppose that  $T_{i,j}=T_{i,i+1}T_{i+1,i+2} \cdots T_{j-1,j}$ and $\tilde{T}_{i,j}=\tilde{T}_{i,i+1}\tilde{T}_{i+1,i+2} \cdots \tilde{T}_{j-1,j}$ for $1\leq i<j\leq n$. $T$ is similar to $\tilde{T}$ if and only if  $X_{i,i}T_{i,i}=\tilde{T}_{i,i}X_{i,i}$ and $X_{i,i}T_{i,j}=\tilde{T}_{i,j}X_{j,j}$, where $X_{i,i}\in \mathcal{L}(\mathcal{H}_i,\tilde{\mathcal{H}}_i)$ is an invertible linear operator for $1\leq i\leq n$. 
\end{cor}
\begin{proof}
Proof of sufficient part follows easily. We will sketch here proof of necessary part. By Lemma \ref{J3}, there exist invertible operators $X_{1,1},X_{2,2},\ldots,X_{n,n}$ such that $$X_{i,i}T_{i,i}=\tilde{T}_{i,i}X_{i,i},\;\;\;1\leq i\leq n$$ and $$X_{i,i}T_{i,i+1}=\tilde{T}_{i,i+1}X_{i+1,i+1},\;\;\; 1\leq i\leq n-1.$$
For $1\leq i<j\leq n$, it is easy to see that
$$X_{i,i}T_{i,j}=\tilde{T}_{i,j}X_{j,j}.$$
\end{proof}

We state and prove a result which convert the problem of finding invertible and $U+K$  intertwining between given two $B_1(\Omega)$ operators to that of finding bounded linear operator with a relation in terms of the curvature of given $B_1(\Omega)$ operators. Let $\pi: E\to \Omega$ and $\tilde{\pi}:\tilde{E}\to \Omega$ be a vector bundles. Set $\mathcal{H}:=\overline{\mbox{span}}\,\{\pi^{-1}(w):w\in\Omega\}$ and $\tilde{\mathcal{H}}:=\overline{\mbox{span}}\,\{\tilde{\pi}^{-1}(w):w\in\Omega\}$. We say that a bundle map $\Phi: E\to \tilde{E}$ is a \textit{bounded bundle map} if $\mathcal{H}$ and $\tilde{\mathcal{H}}$ are Hilbert spaces and $\Phi$ induces a bounded linear map from $\mathcal{H}$ to $\tilde{\mathcal{H}}$.

\begin{prop} \label{sim}
Let $T, \tilde{T}\in B_1(\Omega)$. Let $\mathcal{A}(\mathcal{H})=\mathcal{L}(\mathcal{H})/\mathcal{K}(\mathcal{H}) $ denote the Calkin algebra,  $\pi: \mathcal{L}(\mathcal{H}) \rightarrow \mathcal{A}(\mathcal{H})$. Suppose that  $\mathcal{K}_{T}(w)-\mathcal{K}_{\tilde{T}}(w)=\frac{\partial^2}{\partial w\partial\overline{w}}ln \Psi(w), w\in \Omega.$ Then we have the following statements:

(1) $T$ is unitarily equivalent to $\tilde{T}$ if and only if $\Psi(w)=|\phi(w)|^2,$ for some holomorphic function $\phi$ on $\Omega$.

(2) $T$ is similar to $\tilde{T}$ if and only if $$\Psi(w)=\frac{\|\Phi(t(w))\|^2}{\|t(w)\|^2}+1,$$ where $E$ is  a Hermitian holomorphic line bundle,  $\Phi: E_{T}\rightarrow E$  is a bounded bundle map and  $t$ is a non zero section of the bundle $E_T$.  

(3) $T\sim_{U+K}\tilde{T}$ if and only if  there  exists a bounded linear operator $X$  such that $\pi(X)=\alpha [I]$, $1>\alpha>0$ and 
$$\Psi(w)=ln(\frac{\|X(t(w))\|^2}{\|t(w)\|^2}+(1-\alpha^2))$$
where  $t$ is a non zero section of $E_T$.  

\end{prop}

\begin{proof} 

First, the statement (1) is well known. Then we only need to prove statement (2) and (3). For statement (2), assume that $T$ is similar to $\tilde{T}$, that is, there exists a bounded invertible operator $Y$ such that $TY=Y\tilde{T}$. Without loss of generality, we can assume that $Y^*Y-1\geq 0$. Otherwise, we can choose some $kY$ instead of $Y$ for some $k>0$. Thus there exists a bounded linear operator $X$ such that $Y^*Y=1+X^*X$. Since $TY=Y\tilde{T}$ and $Y$ is invertible, so $Y(t(\cdot))$ is a non zero section of $E_{\tilde{T}}$. For $w$ in $\Omega$, we have 

\begin{eqnarray*}
\mathcal{K}_{\tilde{T}}(w)&=&-\frac{\partial^2}{\partial w\partial\overline{w}}ln(\|Y(t(w)\|^2)\\
&=&-\frac{\partial^2}{\partial w\partial\overline{w}}ln(\|X(t(w))\|^2+\|t(w)\|^2).
\end{eqnarray*}
Thus we have
$$\mathcal{K}_{T}(w)-\mathcal{K}_{\tilde{T}}(w)=\frac{\partial^2}{\partial w\partial\overline{w}}ln\big(\frac{\|X(t(w))\|^2}{\|t(w)\|^2}+1\big).$$  Let $\Phi: E\to \tilde{E}$ be the bounded bundle map which induces the bounded operator $X$ then 
this finishes the proof of necessary part.

From the given condition, there exists a bounded operator $X$ such that 

\begin{eqnarray*} \mathcal{K}_{\tilde{T}}(w)&=&-\frac{\partial^2}{\partial w\partial\overline{w}}ln(\|X(t(w))\|^2+\|t(w)\|^2)\\
&=&-\frac{\partial^2}{\partial w\partial\overline{w}}ln(\la (1+X^*X)(t(w)), t(w)\ra)\\
&=&-\frac{\partial^2}{\partial w\partial\overline{w}}ln(\|(1+X^*X)^{1/2}(t(w))\|^2)\\
\end{eqnarray*}
Set $Y:=(1+X^*X)^{1/2}$. Clearly, $Y$ is an invertible operator and $\mathcal{K}_{\tilde{T}}(w)=\mathcal{K}_{YTY^{-1}}(w)$, $w\in \Omega$.  Thus  $\tilde{T}$ is unitarily equivalent to $ YTY^{-1}$ and hence  $\tilde{T}$ is similar to $T$.

Now we give the proof of statement (3).  
Suppose that $Y=U+K$ is an invertible operator and $YT=\tilde{T}Y$, where $U$ is a unitary operator and $K$ is a compact operator.  Set $\tilde{K}:=U^*K$, it is easy to see that $Y^*Y=I+\tilde{K}+\tilde{K}^*+\tilde{K}^*\tilde{K}$. Set $G:=\tilde{K}+\tilde{K}^*+\tilde{K}^*\tilde{K}$. Since $G$ is self-adjoint compact operator, so there exists $\{\lambda_k\}^{\infty}_{k=1}$ such that $$G=\lambda_1I_{n_1}\oplus \lambda_2I_{n_2}\oplus \cdots\oplus\lambda_kI_{n_k}\oplus\cdots,$$
where $\mbox{dim}(\ker\,(G-\lambda_{k}))=n_k>0$ and $\lim\limits_{k\rightarrow \infty} \lambda_k=0$.   Since $I+G$ is a positive operator and $\lim\limits_{k\rightarrow \infty} \lambda_k=0$, then there exists $1>\alpha>0$ such that $\alpha I+G$ is positive operator.  Now set $$K_1:=\bigoplus\limits_{k=1}^{\infty}((\alpha-\lambda_k)^{\frac{1}{2}}-\alpha^{\frac{1}{2}})I_{n_k}.$$ Then $K_1$ is a compact operator. By a direct calculation,  we have that $$(\alpha^{\frac{1}{2}}I+K_1)^*(\alpha^{\frac{1}{2}}I+K_1)=\alpha I +G.$$
Set $X:=\alpha^{\frac{1}{2}}I+K_1$, then we have that $\pi(X)=\alpha^{\frac{1}{2}}[I]$ and $X^*X+(1-\alpha)=(1+\tilde{K})^*(1+\tilde{K}).$
It follows that 
$$\mathcal{K}_{T}(w)-\mathcal{K}_{\tilde{T}}(w)=\frac{\partial^2}{\partial w\partial\overline{w}}ln(\frac{\|X(t(w))\|^2}{\|t(w)\|^2}+(1-\alpha)).$$
 This finishes the proof of necessary part. The sufficient part will follows from the same argument above.

\end{proof}

\begin{rem} R. G. Douglas, H. Kwon and S. Treil proved that for any 
$n$-hypercontraction  $T \in B_1(\mathbb{D})$,  T is similar to  
$S^*_{z}$
if and only if there exists a bounded subharmonic function 
 $\Psi$  defined on
$\mathbb{D}$ such that $\mathcal{K}_T-\mathcal{K}_{S^*_z}=\frac{\partial^2}{\partial w\partial\overline{w}}\Psi.$ In proposition \ref{sim}, we gave a concrete description of the function  $\Psi$. In the following, we will point out that $\Psi$ is also a bounded subharmonic function. 

 Since $\Phi$ is bounded, 
it is easy to see that $\Psi(w)=ln(\frac{\|\Phi(t(w))\|^2}{\|t(w)\|^2}+1)$ is bounded function. When $T$ is  $n$-hypercontraction, by the operator model theorem, there exists a holomorphic bundle $\mathcal{E}$ such that  $$E_{T}=E_{S^*_z}\otimes \mathcal{E}, $$
where $\mathcal{E}(w)=\bigvee\{D_{T}(t(w)\},$ $D_{T}:=\sum\limits_{k=0}^{n}(-1)^{k}{n\choose k} T^{*k}T^k$, and $t$ is a non zero section of $E_T$. Thus, we have that 
$$\mathcal{K}_T(w)-\mathcal{K}_{S^*_z}(w)=\mathcal{K}_{\mathcal E} (w)=-\frac{\partial^2}{\partial w\partial\overline{w}}ln(\|D_{T}(t(w))\|^2).$$ 
Notice that $\la D(T)(t(z)), D(T)(t(w))\ra$ is  a non negative definite reproducing kernel, so we have that $$-\frac{\partial^2}{\partial w\partial\overline{w}}ln(\|D_{T}(t(w))\|^2)<0,$$ then it follows that 
$$\frac{\partial^2}{\partial w\partial\overline{w}}ln(\frac{\|\Phi(t(w))\|^2}{\|t(w)\|^2}+1)=\mathcal{K}_{S^*_z}(w)-\mathcal{K}_T(w)=\frac{\partial^2}{\partial w\partial\overline{w}}ln(\|D_{T}(t(w))\|^2)>0.$$


\end{rem}

\begin{cor} Let $T$ and $\tilde{T}$ be operators in $B_1(\Omega)$. Suppose that $\{T\}^{\prime}\cong \mathcal{H}^{\infty}(\Omega)$. If there exists $\phi_i\in \mathcal{H}^{\infty}(\Omega)$, $i=1,2$ such that 
$$\mathcal{K}_{T}(w)-\mathcal{K}_{\tilde{T}}(w)=\frac{\partial^2}{\partial w\partial\overline{w}}ln(|\phi_1(w)|^2+|\phi_2(w)|^2),\;\; w\in \Omega,$$
then $T$ is similar to $\tilde{T}$. 
\end{cor}

\begin{rem}  For any two Cowen-Douglas operators, it is always not easy to decide when an intertwining operator (or a holomorphic bundle map) is invertible or not. 
Thus, before getting such kind of invertible bundle map, it is natural that one should find the bounded bundle map first. Thus, Proposition \ref{sim} gives a way to describe 
the similarity of two operators in $B_1(\Omega)$ by searching for the bounded bundle map to match with the difference of curvatures.  For the $U+K$ similarity case, by using 
Proposition  \ref{sim}, we see that the bounded operator appear in the difference of curvature can also be in the form of a unitary operator plus a compact operator but may not be invertible.

\end{rem}

Now we state and prove one of the main theorem of the paper. 

\begin{thm}\label{Main1}
Let $T$ and $\tilde{T}$ be operators in $\mathcal{CFB}_n(\Omega)$.
  $T$ is similar to $\widetilde{T}$ if and only if  the following statements hold:
 \begin{enumerate}
\item[(1)] $\mathcal{K}_{T_{i,i}}-\mathcal{K}_{\tilde{T}_{i,i}}=\frac{\partial^2}{\partial w\partial\overline{w}}ln(\phi_i)$, $1\leq i\leq n$;
\item[(2)]
$\frac{\phi_i}{\phi_{i+1}}\theta_{i,i+1}(T)=\theta_{i,i+1}(\tilde{T})$, $1\leq i\leq n-1$
 \end{enumerate} 
 where $\phi_i=\frac{\|\Phi_i(t_i)\|^2}{\|t_i\|^2}+1$, $\Phi_i: {E_{T_{i,i}}}\to E_i$ is a bounded bundle map, $t_i$ is a non zero section of bundle $E_{T_{i,i}}$, $E_i$ is a Hermitian holomorphic line bundle for $1\leq i\leq n$. 
\end{thm}
\begin{proof}

Suppose conditions $(1)$ and $(2) $ are satisfied. By proposition  \ref{sim}, there exist invertible operators $X_1,X_2,\ldots,X_n$ such that 
$$T_{i,i}=X_i\tilde{T}_{i,i}X_i^{-1},\;\; i=1,2,\ldots,n-1.$$

Let $\overline{T}$ denote the following operator 
$$\overline{T}=\left ( \begin{smallmatrix}\tilde{T}_{1,1}&X_1T_{1,2}X_2^{-1}&\cdots& \cdots&X_1T_{1,n}X_n^{-1}\\
0&\tilde{T}_{2,2}&X_2T_{2,3}X_3^{-1}&\cdots&X_2T_{2,n}X_n^{-1} \\
\vdots&\ddots&\ddots&\ddots&\vdots\\
0&0&\cdots&\tilde{T}_{n-1,n-1}&X_{n-1}T_{n-1,n}X^{-1}_n\\
0&0&\cdots&0&\tilde{T}_{n,n}
\end{smallmatrix}\right ).$$

By definition of $\overline{T}$, it follows that $\overline{T}$ is similar to $T$.

Set 
$$A:=\left ( \begin{smallmatrix}\tilde{T}_{1,1}&X_1T_{1,2}X_2^{-1}&0& \cdots&0\\
0&\tilde{T}_{2,2}&X_2T_{2,3}X_3^{-1}&0 &0 \\
\vdots&\ddots&\ddots&\ddots&\vdots\\
0&0&\cdots&\tilde{T}_{n-1,n-1}&X_{n-1}T_{n-1,n}X^{-1}_n\\
0&0&\cdots&0&\tilde{T}_{n,n}
\end{smallmatrix}\right )\;\mbox{and}\;B:=\left ( \begin{smallmatrix}\tilde{T}_{1,1}&\tilde{T}_{1,2}&0& \cdots&0\\
0&\tilde{T}_{2,2}&\tilde{T}_{2,3}&0&0 \\
\vdots&\ddots&\ddots&\ddots&\vdots\\
0&0&\cdots&\tilde{T}_{n-1,n-1}&\tilde{T}_{n-1,n}\\
0&0&\cdots&0&\tilde{T}_{n,n}
\end{smallmatrix}\right ).$$
It is easy to see that 
  $$\theta_{i,i+1}(\overline{T})=\theta_{i,i+1}(A),  \theta_{i,i+1}(\tilde{T})=\theta_{i,i+1}(B), \;\; 1\leq i\leq n-1.$$

We claim that $A$ is unitarily equivalent to $B$.  In fact, by Theorem 3.6 in \cite{JJKM},  we only need to prove the second fundamental form of $A$ and $B$ are same. Clearly $X_{i+1}(t_{i+1}(\cdot))$ is a non zero section of $E_{\tilde{T}_{i+1,i+1}}$, 
\begin{eqnarray*}
\theta_{i,i+1}(A)(w)&=&\frac{\|X_iT_{i,i+1}X_{i+1}^{-1}X_{i+1}(t_{i+1}(w))\|^2}{\|X_{i+1}(t_{i+1}(w))\|^2}\\
&=&\frac{\|X_iT_{i,i+1}(t_{i+1}(w))\|^2}{\|X_{i+1}(t_{i+1}(w))\|^2}\\
\end{eqnarray*}
and 
$$\theta_{i,i+1}(B)(w)=\frac{\|\tilde{T}_{i,i+1}X_{i+1}(t_{i+1}(w))\|^2}{\|X_{i+1}(t_{i+1}(w))\|^2}.$$

Since $$\phi_i(w)=\frac{\|X_{i}(t_{i}(w))\|^2}{\|t_{i}(w)\|^2},\; 1\leq i\leq n,$$
so we have

\begin{eqnarray*}
\theta_{i,i+1}(B)&=&\theta_{i,i+1}(\tilde{T})\\
&=&\frac{\phi_i(w)}{\phi_{i+1}(w)}\theta_{i,i+1}(T)\\
&=&\frac{\|X_{i}(T_{i,i+1}(t_{i+1}(w)))\|^2}{\|T_{i,i+1}(t_{i+1}(w))\|^2}\frac{\|t_{i+1}(w)\|^2}{\|X_{i+1}(t_{i+1}(w))\|^2}\frac{\|T_{i,i+1}(t_{i+1}(w))\|^2}{\|t_{i+1}(w)\|^2}\\
&=&\frac{\|X_iT_{i,i+1}(t_{i+1}(w))\|^2}{\|X_{i+1}(t_{i+1}(w))\|^2}\\
&=&\theta_{i,i+1}(A).
\end{eqnarray*}

Thus there exists a diagonal unitary operator $V=\mbox{diag}\,\{V_1,V_2,\cdots,V_n\}$ such that $V^*BV=A$. Set $X:=\mbox{diag}\,\{X_1,X_2,\cdots, X_n\}$, consider
\begin{eqnarray*}
VXTX^{-1}V^*
&=&V\overline{T} V^*\\
&=&\left ( \begin{smallmatrix}\tilde{T}_{1,1}&\tilde{T}_{1,2}& \mathsmaller{V_1X_1T_{1,3}X^{-1}_3V^*_3}&\mathsmaller{V_1X_1T_{1,4}X^{-1}_4V^*_4}& \cdots&\mathsmaller{V_1X_1T_{1,n}X_n^{-1}V^*_n}\\
0&\tilde{T}_{2,2}&\tilde{T}_{2,3}&\mathsmaller{V_2X_2T_{2,4}X_4V^*_4}&\cdots&\mathsmaller{V_2X_2T_{2,n}X_n^{-1}V^*_n} \\
0&0&\tilde{T}_{3,3}&\tilde{T}_{3,4}&\cdots&\mathsmaller{V_3X_3T_{3,n}X_n^{-1}V^*_n} \\
\vdots&\vdots&\vdots&\ddots&\ddots&\vdots\\
0&0&\cdots&\tilde{T}_{n-2,n-2}&\tilde{T}_{n-2,n-1}&\mathsmaller{V_{n-2}X_{n-2}T_{n-2,n}X^{-1}_{n}V^*_n}\\
0&0&0&\cdots&\tilde{T}_{n-1,n-1}&\tilde{T}_{n-1,n}\\
0&0&0&\cdots&0&\tilde{T}_{n,n}
\end{smallmatrix}\right ).
\end{eqnarray*}
By lemma \ref{J21}, there exist a bounded operator $K$ such that  $I+K$ is invertible and 
$$(I+K)VXTX^{-1}V^*(I+K)^{-1}=\tilde{T}.$$
 Thus $T$ is similar to $\tilde{T}$.

On the other hand, suppose that  $T$ is similar to $\tilde{T}$, that is, there is an  invertible linear operator $X$ such that $TX=X\tilde{T}$. By Proposition \ref{pjjkm}, $X$ is upper triangular. By Lemma \ref{J3}, we have 
\begin{eqnarray*}
&&\left ( \begin{smallmatrix}X_{1,1} & 0& 0& \cdots & 0\\
0&X_{2,2}&0&\cdots&0\\
\vdots&\ddots&\ddots&\ddots&\vdots\\
0&\cdots&0&X_{n-1,n-1}&0\\
0&\cdots&\cdots&0&X_{n,n}
\end{smallmatrix}\right )\left ( \begin{smallmatrix}T_{1,1} & T_{1,2}& 0& \cdots & 0\\
0&T_{2,2}&T_{2,3}&\cdots&0\\
\vdots&\ddots&\ddots&\ddots&\vdots\\
0&\cdots&0&T_{n-1,n-1}&0\\
0&\cdots&\cdots&0&T_{n,n}
\end{smallmatrix}\right )\\
&=&
\left ( \begin{smallmatrix}\tilde{T}_{1,1} & \tilde{T}_{1,2}&0& \cdots & 0\\
0&\tilde{T}_{2,2}&\tilde{T}_{2,3}&\cdots&0\\
\vdots&\ddots&\ddots&\ddots&\vdots\\
0&\cdots&0&\tilde{T}_{n-1,n-1}&0\\
0&\cdots&\cdots&0&\tilde{T}_{n,n}
\end{smallmatrix}\right )
\left ( \begin{smallmatrix}X_{1,1} & 0& 0& \cdots & 0\\
0&X_{2,2}&0&\cdots&0\\
\vdots&\ddots&\ddots&\ddots&\vdots\\
0&\cdots&0&X_{n-1,n-1}&0\\
0&\cdots&\cdots&0&X_{n,n}
\end{smallmatrix}\right ).
\end{eqnarray*}
It follows that 
 $$X_{i,i}T_{i,i}=\tilde{T}_{i,i}X_{i,i},\;\; X_{i,i}T_{i,i+1}=\tilde{T}_{i,i+1}X_{i+1,i+1}.$$

Set  $$\phi_i(w):=\frac{\|X_{i}(t_{i}(w))\|^2}{\|t_{i}(w)\|^2},\;\; 1\leq i\leq n.$$

By following  the same argument as in the sufficient part and from Proposition \ref{sim}, we can conclude that  
$$\mathcal{K}_{T_{i,i}}-\mathcal{K}_{\tilde{T}_{i,i}}=\frac{\partial^2}{\partial w\partial\overline{w}}ln(\phi_i)$$
and 
$$\frac{\phi_i}{\phi_{i+1}}\theta_{i,i+1}(T)=\frac{\|X_iT_{i,i+1}(t_{i+1})\|^2}{\|X_{i+1}(t_{i+1})\|^2}=\frac{\|\tilde{T}_{i,i+1}X_{i+1}(t_{i+1})\|^2}{\|X_{i+1}(t_{i+1})\|^2}=\theta_{i,i+1}(\tilde{T}).$$
This finishes the proof of the necessary part. 

\end{proof}

\section{Applications}

In this section, we construct uncountably many (non-similar) strongly irreducible operators in $B_n(\mathbb{D})$. We also characterize a class of strongly irreducible weakly homogeneous operators in $B_n(\mathbb{D})$.
Let $M_{i,z}$ be the multiplication operator on a reproducing kernel Hilbert space $\mathcal{H}_{K_i}$, of holomprohic functions defined on $\mathbb{D}$, $1\leq i\leq n$. Suppose $\{M_{i,z}\}^{\prime}=H^{\infty}(\mathbb{D})$ and $M_{i,z}^*\in B_1(\mathbb{D}), 1\leq i\leq n$.
Set

$$T:=\left ( \begin{smallmatrix}M_{1,z}^* & M_{\phi_{1,2}}^*& M_{\phi_{1,3}}^*& \cdots & M_{\phi_{1,n}}^*\\
&M_{2,z}^*&M_{\phi_{2,3}}^*&\cdots&M_{\phi_{2,n}}^*\\
&&\ddots&\ddots&\vdots\\
&&&M_{n-1,z}^*&M_{\phi_{n-1,n}}^*\\
&&&&M_{n,z}^*
\end{smallmatrix}\right ),\tilde{T}:=\left ( \begin{smallmatrix}M_{1,z}^* & M_{\widetilde{\phi}_{1,2}}^*& M_{\widetilde{\phi}_{1,3}}^*& \cdots & M_{\widetilde{\phi}_{1,n}}^*\\
&M_{2,z}^*&M_{\widetilde{\phi}_{2,3}}^*&\cdots&M_{\widetilde{\phi}_{2,n}}^*\\
&&\ddots&\ddots&\vdots\\
&&&M_{n-1,z}^*&M_{\widetilde{\phi}_{n-1,n}}^*\\
&&&&M_{n,z}^*
\end{smallmatrix}\right ) $$
and 

$$T_1:=\left ( \begin{smallmatrix}
 M_{1,z}^* & M_{\phi_{1,2}}^* &0 & \cdots &\cdots &0\\
&M_{2,z}^*&M_{\phi_{2,3}}^*& 0&\cdots &0 \\
&&\ddots&\ddots & \ddots & \vdots\\
&&&M_{n-2,z}^* & M_{\phi_{n-2,n-1}}^* & 0\\
&&&& M_{n-1,z}^* & M_{\phi_{n-1,n}}^*\\
&&&&& M_{n,z}^*
\end{smallmatrix} \right ), 
\tilde{T}_1:=\left ( \begin{smallmatrix}
 M_{1,z}^* & M_{\widetilde{\phi}_{1,2}}^* &0 & \cdots &\cdots &0\\
&M_{2,z}^*&M_{\widetilde{\phi}_{2,3}}^*& 0&\cdots &0 \\
&&\ddots&\ddots & \ddots & \vdots\\
&&&M_{n-2,z}^* & M_{\widetilde{\phi}_{n-2,n-1}}^* & 0\\
&&&& M_{n-1,z}^* & M_{\widetilde{\phi}_{n-1,n}}^*\\
&&&&& M_{n,z}^*
\end{smallmatrix} \right )$$

Assume that  $\phi_{i,i+1},  \tilde{\phi}_{i,i+1}, 1\leq i\leq n-1$, are all non-zero bounded holomorphic  functions. Then by Lemma \ref{SI},  it follow that $T$, $\tilde{T}$, $T_1$ and $\tilde{T}_1$ are all strongly irreducible operators. 

\begin{prop} Let $T$ and $\tilde{T}$
 be elements of $\mathcal{CFB}_n(\mathbb{D})$.
 Then $T$ is similar to $\tilde{T}$ if and only if the zeros, along with its multiplicity, of $\phi_{i,i+1}$ and $\widetilde{\phi}_{i,i+1}$ are same, and both $\frac{\phi_{i,i+1}}{\widetilde{\phi}_{i,i+1}}$ and $\frac{\widetilde{\phi}_{i,i+1}}{\phi_{i,i+1}}$ are elements of $H^{\infty}(\mathbb{D}), 1\leq i\leq n-1$.
\end{prop}
\begin{proof} By Lemma \ref{J21} and Lemma \ref{J3}, $T$ is similar to $\tilde{T}$  if and only if $T_1$ is similar to $\tilde{T}_1$.
First, assume that given conditions are satisfied. We need to show that $T_1$ is similar to $\tilde{T}_1$.
Set $\psi_{i,i+1}:= \frac{\phi_{i,i+1}}{\widetilde{\phi}_{i,i+1}}, 1\leq i\leq n-1 $, $X_1:=M_{\psi_{1,2}}^*M_{\psi_{2,3}}^*\cdots M_{\psi_{n-1,n}}^*$, $X_2:=M_{\psi_{2,3}}^*M_{\psi_{3,4}}^*\cdots M_{\psi_{n-1,n}}^*, \ldots, X_{n-1}:=M_{\psi_{n,n-1}}^*, X_n:=I$. It is easy to see that 
$$X=\left ( \begin{smallmatrix}
X_1 & 0&0 &\ldots &0\\
0 & X_2 & 0& \ldots&0\\
\vdots&\vdots&\ddots& \ddots & \vdots\\
0&0&\cdots& X_{n-1} &0\\
0&0&\cdots&0& X_n
\end{smallmatrix} \right )$$ is invertible and $T_1X=X\tilde{T}_1$. 

Conversely, suppose that $T_1$ is similar to $\tilde{T}_1$. So, by Lemma \ref{J3}, there exist invertible operators $X_1,X_2,\ldots, X_n$ such that $X_i{M_{i,z}}^*={M_{i,z}}^*X_i, 1\leq i\leq n,$ and $X_iM_{\phi_{i,i+1}}^*=M_{\widetilde {\phi}_{i,i+1}}^*X_{i+1}, 1\leq i\leq n-1$. From commuting relation $X_i{M_{i,z}}^*={M_{i,z}}^*X_i$ and  $\{M_{i, z}\}^{\prime}=H^{\infty}(\mathbb{D})$ , there exists a $\psi_i$ in $H^{\infty}(\mathbb{D})$ such that $X_i=M_{\psi_i}^*$ for $1\leq i\leq n$. Since $X_i$ is invertible, so $\psi_i$ is non zero, and $\frac{1}{\psi_i}$ is bounded holomorphic function. From $X_iM_{\phi_{i,i+1}}^*=M_{\widetilde {\phi}_{i,i+1}}^*X_{i+1}$, we get
$\psi_i(z)\phi_{i,i+1}(z)=\widetilde{\phi}_{i,i+1}(z)\psi_{i+1}(z)$ for all $z$ in $\mathbb{D}$. Thus it follows that the zeros, along with its multiplicity, of the functions $\phi_{i,i+1}$ and $\widetilde{\phi}_{i,i+1}$ are same and both $\frac{\phi_{i,i+1}}{\widetilde{\phi}_{i,i+1}}$ and  $\frac{\widetilde{\phi}_{i,i+1}}{\phi_{i,i+1}}$ are bounded holomorphic functions on $\mathbb{D}$.
\end{proof}

\begin{cor}
 Let $M_{i,z}$ be the multiplication operator on the reproducing kernel Hilbert space $H_{K_i}$, where $K_i(z,w)=\frac{1}{(1-z\bar{w})^{\lambda_i}}$, $z,w\in \mathbb{D}, \,1\leq i\leq n$. Suppose that $1\leq \lambda_i\leq \lambda_{i+1}<\lambda_i+2$ for $1\leq i\leq n-1$. Set
  $$T_{(a_1, a_2,\ldots a_{n-1},m_1,m_2,\ldots,m_{n-1})}:=\left ( \begin{smallmatrix}
 M_{1,z}^* & M_{(z-a_1)^{m_1}}^* & M_{\phi_{1,3}}^* & \cdots &M_{\phi_{1,n-1}}^*& M_{\phi_{1,n}}^*\\
&M_{2, z}^*&M_{(z-a_2)^{m_2}}^*& M_{\phi_{2,4}}^*&\cdots &M_{\phi_{2,n}}^* \\
&&\ddots&\ddots & \ddots & \vdots\\
&&&M_{n-2,z}^* & M_{(z-a_{n-2})^{m_{n-2}}}^* & M_{\phi_{n-2,n}}^*\\
&&&& M_{n-1,z}^* & M_{(z-a_{n-1})^{m_{n-1}}}^*\\
&&&&& M_{n,z}^*
\end{smallmatrix} \right )$$ 
and 
$$\tilde{T}_{(b_1, b_2,\ldots b_{n-1},l_1,l_2,\ldots,l_{n-1})}:=\left ( \begin{smallmatrix}
 M_{1,z}^* & M_{(z-b_1)^{l_1}}^* & M_{\widetilde{\phi}_{1,3}}^* & \cdots &M_{\widetilde{\phi}_{1,n-1}}^*& M_{\widetilde{\phi}_{1,n}}^*\\
&M_{2,z}^*&M_{(z-b_2)^{l_2}}^*& M_{\widetilde{\phi}_{2,4}}^*&\cdots &M_{\widetilde{\phi}_{2,n}}^* \\
&&\ddots&\ddots & \ddots & \vdots\\
&&&M_{n-2,z}^* & M_{(z-b_{n-2})^{l_{n-2}}}^* & M_{\widetilde{\phi}_{n-2,n}}^*\\
&&&& M_{n-1,z}^* & M_{(z-b_{n-1})^{l_{n-1}}}^*\\
&&&&& M_{n,z}^*
\end{smallmatrix} \right ),$$
where $a_i, b_i\in \mathbb{D}$ and $m_i,l_i\in \mathbb{N}$ for $1\leq i\leq n-1$ and 
$$\frac{\phi_{i,j}}{(z-a_i)^{m_i}(z-a_{i+1})^{m_{i+1}}\ldots (z-a_{j-1})^{m_{j-1}}}
, \frac{\widetilde{\phi}_{i,j}}{(z-b_i)^{l_i}(z-b_{i+1})^{l_{i+1}}\ldots (z-b_{j-1})^{l_{j-1}}}$$ are bounded holomorphic functions on  $\mathbb{D}$, $2\leq i<j\leq n, j-i\geq2$. Then $T_{(a_1, a_2,\ldots a_{n-1},m_1,m_2,\ldots,m_{n-1})}$ is similar to $\tilde{T}_{(b_1, b_2,\ldots b_{n-1},l_1,l_2,\ldots,l_{n-1})}$ if and only if
 $a_i=b_i$ and $m_i=l_i$ for $1\leq i\leq n-1$.
 \end{cor}
 
Now we state and prove a result which characterizes a class of weakly homogeneous  operators in $B_n(\mathbb{D})$. This generalizes \cite[Theorem 3.6]{Ghara}.
 
\begin{defn}
A bounded linear operator $T$, defined on a Hilbert space $\mathcal{H}$, is said to be weakly homogeneous if $\sigma(T)\subseteq \bar{\mathbb{D}}$ and $\phi(T)$ is similar to $T$ for all $\phi$ in $\mbox{M}\ddot{\mbox{o}}\mbox{b}$.
\end{defn}
\begin{prop}
Let $M_{i,z}$ be the multiplication operator on the reproducing kernel Hilbert space $H_{K_i}$, where $K_i(z,w)=\frac{1}{(1-z\bar{w})^{\lambda_i}}$, $z,w\in \mathbb{D}, \,1\leq i\leq n$. Suppose that $1\leq \lambda_i\leq \lambda_{i+1}<\lambda_i+2$ for $1\leq i\leq n-1$. Set $$T:=\left ( \begin{smallmatrix}
 M_{1,z}^* & M_{\psi_{1,2}}^* & M_{\phi_{1,3}}^* & \cdots &M_{\phi_{1,n-1}}^*& M_{\phi_{1,n}}^*\\
&M_{2, z}^*&M_{\psi_{2,3}}^*& M_{\phi_{2,4}}^*&\cdots &M_{\phi_{2,n}}^* \\
&&\ddots&\ddots & \ddots & \vdots\\
&&&M_{n-2,z}^* & M_{\psi_{n-2,n-1}}^* & M_{\phi_{n-2,n}}^*\\
&&&& M_{n-1,z}^* & M_{\psi_{n-1,n}}^*\\
&&&&& M_{n,z}^*
\end{smallmatrix} \right )$$
where $\psi_{i,i+1}\in C(\bar{\mathbb{D}})\cap \mbox{Hol}(\mathbb{D}), 1\leq i\leq n-1$, is non-zero and  $\phi_{i,j}\in H^{\infty}(\mathbb{D})$ for $1\leq i<j\leq n$ and $j-i \geq 2$. The operator $T$ is weakly homogeneous if and only if each $\psi_{i,i+1}$, $1\leq i\leq n-1$, is non-vanishing.
\end{prop}
\begin{proof}
Set $$T_1:=\left ( \begin{smallmatrix}
 M_{1,z}^* & M_{\psi_{1,2}}^* &0 & \cdots &\cdots &0\\
&M_{2,z}^*&M_{\psi_{2,3}}^*& 0&\cdots &0 \\
&&\ddots&\ddots & \ddots & \vdots\\
&&&M_{n-2,z}^* & M_{\psi_{n-2,n-1}}^* & 0\\
&&&& M_{n-1,z}^* & M_{\psi_{n-1,n}}^*\\
&&&&& M_{n,z}^*
\end{smallmatrix} \right ).$$
By Corollary \ref{HMPH}, Lemma \ref{J21} and Lemma \ref{J3}, it follows that $T$ is weakly homogeneous if and only if $T_1$ is weakly homogeneous.
First, we show that necessary part, that is, if given conditions are satisfied, then $T_1$ is weakly homogeneous. It suffices to show that $T_1^*$ is weakly homogeneous. To show this, we consider
$X_1=M_{\tfrac{(\psi_{1,2}\circ \phi^{-1})((\phi^{-1})^{\prime}\circ\phi)}{\psi_{1,2}}} M_{\tfrac{(\psi_{2,3}\circ \phi^{-1})((\phi^{-1})^{\prime}\circ\phi)}{\psi_{2,3}}}\cdots M_{\tfrac{(\psi_{n-1,n}\circ \phi^{-1})((\phi^{-1})^{\prime}\circ\phi)}{\psi_{n-1,n}}}C_{\phi^{-1}}, X_2=M_{\tfrac{(\psi_{2,3}\circ \phi^{-1})((\phi^{-1})^{\prime}\circ\phi)}{\psi_{2,3}}}\cdots M_{\tfrac{(\psi_{n-1,n}\circ \phi^{-1})((\phi^{-1})^{\prime}\circ\phi)}{\psi_{n-1,n}}}C_{\phi^{-1}},\ldots,
X_{n-1}=M_{\tfrac{(\psi_{n-1,n}\circ \phi^{-1})((\phi^{-1})^{\prime}\circ\phi)}{\psi_{n-1,n}}}C_{\phi^{-1}}, X_n=C_{\phi^{-1}}$. As each $X_i, 1\leq i\leq n,$  is invertible, so  $$X_{\phi}=\left ( \begin{smallmatrix}
X_1 & 0&0 &\ldots &0\\
0 & X_2 & 0& \ldots&0\\
\vdots&\vdots&\ddots& \ddots & \vdots\\
0&0&\cdots& X_{n-1} &0\\
0&0&\cdots&0& X_n
\end{smallmatrix} \right )$$ is invertible and note that $\phi(T_1^*)X_{\phi}=X_{\phi} T_1^*$. 

Conversely, assume that $T_1$ is weakly homogeneous. So, by Lemma \ref{J3}, there exist invertible operators $X_1,X_2,\ldots, X_n$ such that $X_iM_{i,z}^*={\phi(M_{i,z}^*)}X_i, 1\leq i\leq n,$ and $X_iM_{\psi_{i,i+1}}^*=\phi^{\prime}(M_{i,z}^*)M_{{\psi}_{i,i+1}}^*X_{i+1}, 1\leq i\leq n-1$. Set $t_i(w):=K_{i,\bar{w}}(\cdot)$ and $t_{i,\phi}:=t_i\circ \phi^{-1}$, where $K_{i,\bar{w}}(z)=K_{i}(z,\bar{w})$ for $z,w\in\mathbb{D}$. Since $X_iM_{i,z}^*={\phi(M_{i,z}^*)}X_i$, so there a $\psi_i$ in $H^{\infty}(\mathbb{D})$ such that $X_i(t_i(w))=\psi_i(w)t_{i,\phi}(w)$ for all $w$ in $\mathbb{D}$, $1\leq i\leq n$. For $1\leq i\leq n$, $X_i$ is invertible, so $\psi_i(w)\neq 0$ for all $w$ in $\mathbb{D}$, and $\frac{1}{\psi_{i}}$ is a bounded holomophic function.  From $X_iM_{\psi_{i,i+1}}^*=\phi^{\prime}(M_{i,z}^*)M_{{\psi}_{i,i+1}}^*X_{i+1}$, $1\leq i\leq n-1$, we get 
\begin{eqnarray}\label{ewh}\psi_i(w)\overline{\psi_{i,i+1}(\bar{w})}=\psi_{i+1}(w)\overline{\psi_{i,i+1}(\overline{\phi^{-1}(w)})} \phi^{\prime}(\phi^{-1}(w)),\;\;\;w\in \mathbb{D}.\end{eqnarray}
We claim that $\psi_{i,i+1}, 1\leq i\leq n-1$, is non-vanishing. Suppose on contrary there exists a point $w_0 \in \mathbb{D}$ such that $\psi_{i,i+1}(w_0)=0$. Since M$\ddot{o}$b acts transitively on $\mathbb{D}$, so by equation (\ref{ewh}), it follows that $\psi_{i,i+1}(w)=0$ for all $w\in {\mathbb{D}}$ and hence $\psi_{i,i+1}(w)=0$ for all $w$ in $\bar{\mathbb{D}}$. This contradicts that $\psi_{i,i+1}$ is a non-zero function. Thus $\psi_{i,i+1}(w)\neq 0$ for all $w\in \mathbb{D}$.

Now we show that $\psi_{i,i+1}, 1\leq i\leq n-1$, is non-vanishing on $\mathbb{T}=\{z\in \mathbb{C}:|z|=1\}$.  Replacing $\phi$ by biholomorphic map $z\mapsto e^{i\theta}z$ in equation (\ref{ewh}), we obtain
 \begin{eqnarray}\label{ewh1}\psi_i(w)\overline{\psi_{i,i+1}(\bar{w})}= e^{i\theta}\psi_{i+1}(w)\overline{\psi_{i,i+1}(e^{i\theta}\bar{w})},\;\;\; w\in \mathbb{D},\;\;\theta\in \mathbb{R}.\end{eqnarray} 
  Suppose there exists a point $e^{i\theta_0}$ such that $\psi_{i,i+1}(e^{i\theta_0})=0$. Choose a sequence $\{w_n\}$ in $\mathbb{D}$ such that $w_n \to e^{-i\theta_0}$ as $n\to \infty$. From equation (\ref{ewh1}), we get
 \begin{eqnarray}\label{ewh2}\psi_i(w_n)\overline{\psi_{i,i+1}(\bar{w_n})}= e^{i\theta}\psi_{i+1}(w_n)\overline{\psi_{i,i+1}(e^{i\theta}\bar{w_n})}.\end{eqnarray}
Since $\psi_{i,i+1}\in C(\bar{\mathbb{D}})$ and $\psi_i, \psi_{i+1}$ are bounded above and below on $\mathbb{D}$, so from equation (\ref{ewh2}), as $n\to \infty$, we get $\psi_{i,i+1}(e^{i(\theta+\theta_0)})=0$ for all $\theta \in \mathbb{R}$.  Thus $\psi_{i,i+1}$ is zero on at every point of $\mathbb{T}$, and hence
$\psi_{i,i+1}$ vanishes identically on $\bar{\mathbb{D}}$. This again contradicts  the hypothesis that $\psi_{i,i+1}$ is a non-zero function. This completes the proof.  
 \end{proof}

\end{document}